\documentclass[a4paper,10pt,reqno]{amsart}
\usepackage[utf8]{inputenc}
\usepackage[british]{babel}
\usepackage[left=3 cm, top=2.5 cm, right=3 cm]{geometry}
\usepackage{graphicx}
\usepackage{amssymb}
\usepackage{amsthm}
\usepackage{amsmath}
\usepackage{amsrefs}
\usepackage[mathscr]{euscript}
\usepackage{xfrac}  
\usepackage{mathtools}
\usepackage{nicefrac}
\usepackage{enumerate}
\usepackage[decimalsymbol=comma]{siunitx}
\usepackage{booktabs}
\usepackage{ftcap}
\usepackage{caption}
\usepackage{float}
\usepackage{todonotes}
\usepackage{multirow}
\usepackage[colorlinks=true, linkcolor=red, urlcolor=red,citecolor=red]{hyperref}
\usepackage{microtype}
\usepackage{enumerate}
\usepackage{stmaryrd}
\usepackage{url}
\usepackage{comment}

\renewcommand{\H}{\mathcal{H}}
\renewcommand{\div}{\mathrm{div}}
\renewcommand{\epsilon}{\varepsilon}
\renewcommand{\d}{\mathrm d}

\newcommand{\co}{\mathrm{co}}
\newcommand{\tr}{\mathrm{tr}}
\newcommand{\D}{\mathrm{D}}
\newcommand{\To}{\mathbb{T}}
\newcommand{\T}{\mathrm{T}}
\newcommand{\im}{\mathrm{Im}}
\newcommand{\id}{\mathrm{id}}
\newcommand{\dist}{\mathrm {dist}}

\theoremstyle{definition}
\newtheorem{thm}{Theorem}[section]
\newtheorem*{thm*}{Theorem}%[section]
\newtheorem{defi}[thm]{Definition}
\newtheorem*{defi*}{Definition}%[section]
\newtheorem{lemma}[thm]{Lemma}
\newtheorem*{lemma*}{Lemma}%[section]

\newtheorem*{prop*}{Proposition}%[section]

\newtheorem*{ex*}{Example}%[section]

\newtheorem*{cor*}{Corollary}%[section]
\newtheorem{rem}[thm]{Remark}
\newtheorem*{rem*}{Remark}%[section]

\newtheorem*{rems*}{Remarks}%[section]

\numberwithin{equation}{section}

\begin{document}

\title[Volume preserving mean curvature flows in flat torus]{Volume preserving mean curvature flows near strictly stable sets in flat torus}
\author{Joonas Niinikoski}
\address{University of Jyvaskyla, Department of Mathematics and Statistics
P.O.Box (MaD)FI-40014 University of Jyvaskyla, Finland.}
\email{joonas.k.niinikoski@jyu.fi}
\subjclass[2010]{Primary 53C44. Secondary 35K93}
\keywords{Periodic stability, strictly stable sets, volume preserving mean curvature flow}

\begin{abstract}
In this paper we establish a new stability result for smooth volume preserving mean curvature flows in flat torus \(\To^n\) 
in dimensions \(n=3,4\). The result says roughly that if an initial set is near to a strictly stable 
set in \(\To^n\) in \(H^3\)-sense, then the corresponding flow has infinite lifetime and
converges exponentially fast to a translate of the strictly stable (critical) set in \(W^{2,5}\)-sense.
\end{abstract}

\maketitle
\tableofcontents

%%%%%%%%%%%%%%%%%%%%%%%%%%%%%%%%%%%%%%%%%%%%%%%%%%%%%%%%%%%%%%%%%%%%%%%%%%%

\section{Introduction}
A smooth evolution of sets \((E_t)_t\), that is, a smooth flow in \(\mathbb R^n\) is a \emph{volume preserving mean curvature flow} (VMCF), if for every time 
\(t\) the (outer) normal velocity of the flow on \(\partial E_t\) obeys the law
\[
V_t = \bar H_t - H_t,
\]
where \(H_t\) is the (scalar) \emph{mean curvature} on \(\partial E_t\) and \(\bar H_t\) its integral average over \(\partial E_t\). As the name suggests, a VMCF
preserves the volume, which is in the contrast to a classical \emph{mean curvature flow} with a smooth and bounded initial set. Such a flow shrinks the initial volume to zero in finite time.

The short time existence for a smooth VMCF in \(\mathbb R^n\)  is well-known. For any smooth (a closed set with a smooth boundary) and compact set 
\(E \subset \mathbb R^n\) with \(n\geq 2\) there is a unique VMCF starting from \(E\).  However, a VMCF \((E_t)_t\) may develop \emph{singularities} 
such as self-intersections of the boundary \(\partial E_t\)  within a finite time, see \cite{MS}. Another type of singularities of VMCF in a free boundary setting is studied in \cite{At}, where it is shown that certain thin necks have to pinch-off under VMCF. A natural problem is to find a sufficient condition for the initial set such that the VMCF starting from the set does not form singularities and has infinite lifetime.

Several contributions concerning the previous question have been made over the years. The classical result of Huisken \cite{Hu} (see also Gage \cite{Ga} in the case \(n=2\)) says that for any smooth, compact and convex set \(E \subset \mathbb R^n\) there is a unique VMCF \((E_t)_t\) starting from \(E\) such that the flow has infinite lifetime and converges exponentially fast  to a closed ball of the volume \(|E|\) in \(C^\infty\)-topology. In the dimension $n=3$, the result by Li \cite{Li} gives us an alternative condition for a connected initial boundary based on a certain energy such that the corresponding VMCF has infinite lifetime and converges exponentially fast to a ball. Notice that if VMCF converges in \(C^2\)-sense, then the limit set is a finite union of balls with mutually positive distance. This follows from the \emph{Alexandrov theorem}.

This naturally raises questions about the stability of VMCF near \emph{stable sets}, in this case the closed Euclidean balls.
Such problems are often called \emph{stability problems}. Escher and Simonett \cite{ES} used to center manifold analysis to prove that if \(E \subset \mathbb R^n\) is a smooth compact set and \(\bar B(x,r)\) is a closed ball with the same volume such that \(\partial E\) is \(C^{1,\alpha}\)-close to \(\partial B(x,r)\), then the VMCF starting from \(E\) has infinite lifetime and converges
to a translate of \(\bar B(x,r)\) exponentially fast  in \(C^k\)-sense for any 
\(k \in \mathbb N\).

Instead of having generic smooth sets in \(\mathbb R^n\), we may focus on periodic smooth sets, that is, the smooth sets in \(\mathbb R^n\) invariant under the lattice translations. This again leads us to consider the \emph{flat torus} \(\To^n\) in place of \(\mathbb R^n\).  One motivation for this is that there then are more different types of compact and \emph{critical} sets than in \(\mathbb R^n\). Also, the notion of VMCF generalizes to the flat torus and corresponds to the periodic VMCFs in \(\mathbb R^n\).  We are interested in the subclass of compact and critical sets in \(\To^n\) called \emph{strictly stable} sets (with respect to perimeter), 
see Definition \ref{StricSet}. Examples of strictly stable sets (besides single balls, cylinders and strips) in \(\To^3\) are those sets having a \emph{Schwarz surface} as a boundary, see \cite{Ro}. Acerbi, Fusco and Morini prove that strictly stable sets in \(\To^n\) are always isolated \emph{local perimeter minimizers} (under the notion of volume in \(\To^n\)), see \cite[Theorem 1.1]{AFM}. In contrast, the only smooth local perimeter minimizers in \(\mathbb R^n\) are just the single balls. This essentially follows from \cite{BdCE}, see also \cite{We}.

Our goal is to prove the following stability result for VMCFs near strictly stable sets in the flat torus \(\To^n\) with \(n=3,4\) using the notion of \emph{graph surface representation in normal direction} and Sobolev spaces on smooth compact hypersurfaces.

\begin{thm*}[\textbf{Main result}]
\ \\
Let \(F \subset \To^n\), where \(n=3,4\), be a strictly stable set and let \(\nu_F\) be the unit normal of \(\partial F\) with the inside-out 
orientation. There exists a positive constant \(\delta_0 \in \mathbb R_+\) depending on \(F\) such that the following hold.

If \(E_0\) is a smooth set in \(\mathbb T^n\) with the same volume as \(F\) and having a boundary of the form
\[
\partial E_0 = \{ x + \psi_0(x)\nu_F(x): x \in \partial F\},
\]
where \(\psi_0 \in C^\infty(\partial F)\) and \(\|\psi_0\|_{H^3(\partial F)} \leq \delta_0 \), then the VMCF
\((E_t)_t\) starting from \(E_0\) has infinite lifetime and converges to a translate \(F+p\) of \(F\) exponentially fast in \(W^{2,5}\)-sense from the point of view
of the boundary \(\partial(F+p)\). Moreover, \(|p| \rightarrow 0\) as \( \|\psi_0\|_{H^3(\partial F)} \rightarrow 0\).
\end{thm*}

\begin{rems*}
\ \\
Of course, using the same arguments we obtain a similar result in \(\To^2\).
Since the convergence happens exponentially fast in time, the flow is also said to be \emph{exponentially stable} near 
strictly stable set.
\end{rems*}

In terms of methods we are motivated by the paper Acerbi, Fusco, Julin and Morini \cite{AFJM}, where they prove similar kinds of stability results 
for other volume preserving flows, namely the \emph{modified Mullins-Sekerka flow} and the \emph{surface diffusion flow} in the three dimensional flat torus \(\To^3\).
The cornerstone of our analysis (see Section 3) is to prove that \(H^3\)-closeness to a strictly stable set implies for the VMCF \((E_t)_t\)
that the \(L^2\)-norm of the normal velocity over \(\partial E_t\), that is \(\|\bar H_t - H_t\|_{L^2(\partial E_t)}\), is decaying exponentially in time while the \(L^2\)-norm of its tangential gradient over \(\partial E_t\) is bounded in time. In proving this we are heavily dependent on Sobolev interpolation inequalities, which is the reason
we have to restrict ourselves to low dimensions.

%%%%%%%%%%%%%%%%%%%%%%%%%%%%%%%%%%%%%%%%%%%%%%%%%%%%%%%%%%%%%%%%%%%%%%%%%%%

\section{Preliminaries}

\subsection*{Flat torus}

Recall that for given \(n \geq 2\) the (unit) flat torus \(\To^n\) is defined as the quotient space \(\mathbb R^n / \ \mathbb Z^n\). Here the equivalence relation is given in the obvious way: \(x \sim y\) exactly when \(x-y \in \mathbb Z^n\). Functions \(f: \To^n \rightarrow \mathbb R^d\) (for \(d \in \mathbb N\)) can be (canonically) identified with the class of \(\mathcal D_n\)-periodic maps \(\mathbb R^n \rightarrow \mathbb R^d\), where
\(\mathcal D_n = [0,1[^n\) is the dyadic unit cube. 
The continuous maps \(\phi : \To^n \rightarrow \To^n\) are identified (up to integer translations) with their continuous lifts 
\(\tilde \phi: \mathbb R^n \rightarrow \mathbb R^n\).
Such lifts \(\tilde \phi :\mathbb R^n \rightarrow \mathbb R^n\) are of the form \(\tilde \phi = L + u\), where \(L \in M_n(\mathbb Z)\) and \(u\) is continuous and \(\mathcal D_n\)-periodic. 
Further, the sets in \(\To^n\) can be identified with  
\(\mathcal D_n\)-periodic sets in \(\mathbb R^n\). If \(x \in \To^n\), then for any \(p \in \mathbb R^n\) the notation \(x+p\) means the element
\(q(\tilde x + p) \in \To^n\), where \(\tilde x \in q^{-1}(x)\) and \(q:\mathbb R^n \rightarrow \To^n\) is the quotient map.

We consider \(\mathbb \To^n\) as a smooth and compact manifold where the smooth structure is given via the quotient map \(\mathbb R^n \rightarrow \mathbb R^n / \ \mathbb Z^n\). The natural flat Riemannian metric on \(\To^n\) is induced by the standard Euclidean inner product \(\langle  \ \cdot \ , \ \cdot \ \rangle\)  of \(\mathbb R^n\) via the quotient map. Indeed, one can think 
\(\To^n\) is “locally" the Euclidean \(\mathbb R^n\). A compatible distance \(d_{\To^n} : \To^n \times \To^n \rightarrow [0,\infty[\) is given by
\[
d_{\To^n}(x,y) = \min \{|\tilde x - \tilde y|: \tilde x \in q^{-1}(x), \tilde y \in q^{-1}(y) \}. 
\]
For every Borel set in \(A\) in \(\To^n\) its \(s\)-dimensional Hausdorff measure \(\mathcal H^s(A)\) is defined to be the corresponding Hausdorff-measure of the intersection of the periodic extension and \(\mathcal D_n\).

A function \(f:\To^n \rightarrow \mathbb R^d\) is locally \(C^{k,\alpha}\)-regular at \(x \in \To^n\) exactly when its periodic extension \(\tilde f\) to \(\mathbb R^n\) is locally \(C^{k,\alpha}\)-regular at every representative \(\tilde x\) of \(x\). Thus, we set the \(j\)-th derivative \(\D^j f (x)\) at \(x\) to be \(\D^j \tilde f(\tilde x)\) for every \(1\leq j \leq k\). Again, a continuous map \(\phi : \To^n \rightarrow \To^n\) is \(C^{k,\alpha}\)-regular exactly when its continuous lifts 
$\tilde \phi$ are \(C^{k,\alpha}\)-maps. Then $\phi$ has a unique continuous lift modulo integer translations so we may again see the derivatives of
$\phi$ as the derivatives of its continuous lifts. If a diffeomorphism \(\Phi: \To^n \rightarrow \To^n\) is close enough to the identity map \(\id_{\To^n}\) in the sense that  \(\sup_{\To^n} d_{\To^n} (\Phi,\id_{\To^n})\) is sufficiently small, 
then there is a unique diffeomorphic lift \(\tilde \Phi\) such that \(\tilde \Phi = \id + u\) and \(\sup_{\To^n} d_{\To^n} (\Phi,\id_{\To^n}) = \sup_{\mathbb R^n} |u|\).
In such a case, we set for every \(l \in \mathbb N\) and \(0 \leq \alpha \leq 1\) (for which \(\Phi\) is \(C^{l,\alpha}\)-regular)
\[
\|\Phi - \id_{\To^n}\|_{C^{l,\alpha}(\To^n;\To^n)}=\|u\|_{C^{l,\alpha}(\mathbb R^n;\mathbb R^n)}.
\]
Further, we  denote \(\id_{\To^n}\) by \(\id\) when there is no danger of confusion.

\subsection*{Smooth hypersurfaces}
By a hypersurface, we mean an embedded submanifold of codimension 1 without boundary in our discussion.
A set \(\Gamma \subset \To^n\) is a \(C^{k,\alpha}\)-hypersurface exactly when its 
\(\mathcal D_n\)-periodic extension is a \(C^{k,\alpha}\)-hypersurface in \(\mathbb R^n\). From now on, \(\Sigma\) denotes a smooth and compact hypersurface of $\To^n$. 
Note that then $\Sigma$ has finitely many (path) connected components $\Sigma_1,\ldots,\Sigma_N$ which are in their own right smooth and compact hypersurfaces. 
The shorthand notation \(|\Sigma| = \mathcal H^{n-1}( \Sigma)\) is used from time to time without any further mention.

A function \(f\) belongs to \(C^k(\Sigma;\mathbb R^d)\) if and only if it admits a \(C^k\)-extension to some open neighborhood of \(\Sigma\). 
For every \(x \in \Sigma\) the \emph{geometric tangent space} \(G_x\Sigma\) is defined as a unique \(n-1\)-dimensional subspace of \(\mathbb R^n\) such that \(x + G_x \Sigma\) is the tangent plane of \(\Sigma\) at \(x\). Equivalently we can set \(G_x\Sigma = \D \phi (\mathbb R^{n-1})\), where \(\phi:U \rightarrow \To^n\) is any local parametrization of \(\Sigma\) at \(x\). Moreover, the orthogonal projection
from \(\mathbb R^n\) onto \(G_x\Sigma\) is denoted by \(P_\Sigma (x)\). Then \(P_\Sigma : \Sigma \rightarrow \mathcal L (\mathbb R^n;\mathbb R^n)\) is a smooth map.

For any \(x \in \Sigma\), open neighborhood \(U\) of \(x\) in \(\To^n\) and \(f \in C^1(U;\mathbb R^k)\) the \emph{tangential differential} 
\(\D_\tau f (x):\mathbb R^n \rightarrow \mathbb R^k\) of \(f\) with respect to 
\(\Sigma\)  at \(x\) is given by 
\[
\D_\tau f (x) = \D f (x) P_\Sigma (x).
\]
The definition does not depend on how \(f\) is extended beyond \(\Sigma\). In the case \(k=1\) the dual of \(\D_\tau f (x)\) is called the \emph{tangential gradient} of \(f\) with respect to \(\Sigma\) at \(x\) denoted by \(\nabla_\tau f (x)\). Then we can write
\(\nabla_\tau f (x)= P_\Sigma (x) \nabla f (x)\) so \(\nabla_\tau f(x) \in G_x \Sigma\). Further, in the case \(k=n\) the \emph{tangential divergence} \(\div_\tau f (x)\) of \(f\) with respect to \(\Sigma\) at \(x\) is defined as the trace of \(\D_\tau f (x)\). These operations behave as their
ordinary counterparts in \(\To^n\).

Since \(\Sigma\) has finitely many connected components, then it is also orientable, i.e., it admits a smooth unit normal field \(\nu: \Sigma \rightarrow \mathbb R^n\), where
\(\nu(x) \in N_x\Sigma := G_x\Sigma^\perp\) for every \(x \in \Sigma\). The pair \((\Sigma,\nu)\) is called an \emph{oriented hypersurface} \(\Sigma\) with an orientation \(\nu\). If \(\Sigma'\) is a \(C^1\)-hypersurface and \(\Phi: \Sigma \rightarrow \Sigma'\) is a \(C^1\) diffeomorphism, then we have the \emph{change of variable formula}. For any \(h \in L^1(\Sigma',\mathcal H^{n-1})\)  
\[
\int_{\Sigma'} h \ \d \H^{n-1}  =  \int_\Sigma (h \circ \Phi) J_\tau \Phi \ \d \H^{n-1},
\]
where the \emph{tangential Jacobian} \(J_\tau \Phi \) is given by \(J_\tau \Phi = |\det \D \tilde \Phi| |(\D \tilde \Phi)^{-\T}\nu|\) and \(\tilde \Phi\) is any diffeomorphic extension of \(\Phi\). This is independent of the choice of orientation.
 
The \emph{second fundamental form} \(B (x)\) of \(\Sigma\) associated with the orientation \(\nu\) at \(x\)
can be seen as the linear operator from \(G_x \Sigma\) to itself or equivalently \(\mathbb R^n \rightarrow \mathbb R^n\)  given by
\[
B (x) =\D_\tau \tilde \nu (x),
\]
where \(\tilde \nu\) is any smooth extension of \(\nu\). We use both of these conventions interchangeably. If we use the latter one, then \(N_x \Sigma \subset \ker B (x)\) and \( \im \ B (x) \subset G_x \Sigma\).
The operator is symmetric and its eigenvalues and corresponding eigenspaces on \( G_x \Sigma\) are called \emph{principal values} and \emph{principal directions}. The \emph{mean curvature} \(H (x)\) of \(\Sigma\) associated with the orientation at \(x\) is now defined as the trace of  \(B (x)\) or equivalently the sum of principal values on \( G_x \Sigma\). Again, the maps \(B : \Sigma \rightarrow \mathcal L (\mathbb R^n;\mathbb R^n)\) and \(H :  \Sigma \rightarrow \mathbb R\) are smooth. If \(H\) is constant, then \(\Sigma\) is said to be critical. The Frobenius norm \(|B|\) on \(\Sigma\) does not depend on the used orientation.
We define also the mean curvature vector field \(\mathbf H : \Sigma \rightarrow \mathbb R^n\) by setting \(\mathbf H =- H \nu\). Notice that \(\mathbf H\) is independent of the choice of orientation.
Since \(\Sigma\) is compact, then the \emph{divergence theorem for hypersurfaces} says that for any \(f \in C^1(\To^n,\mathbb R^n)\) 
\[
\int_\Sigma \div_\tau f \ \d \H^{n-1} =-\int_\Sigma \langle f ,\mathbf H  \rangle \ \d \H^{n-1} = \int_\Sigma H \langle f,\nu \rangle \ \d \H^{n-1}.
\]
If \(f(x) \in G_x\Sigma\) for every \(x \in \Sigma\), then the previous formula yields the \emph{integration by parts formula}
\[
\int_\Sigma \varphi \ \div_\tau f \ \d \H^{n-1} = - \int_\Sigma \langle \nabla_\tau \varphi, f \rangle \ \d \H^{n-1}
\] 
for every \(\varphi \in C^1(\Sigma)\), in particular \(\int_\Sigma \div_\tau f \ \d \H^{n-1} =0\).

While the concept of tangential derivative can be defined similarly on compact \(C^1\)-hypersurfaces, the classical notion of mean curvature 
is not generally possible for such surfaces due to lack of regularity. However, the following generalization is possible. We say that a compact  
\(C^1\)-hypersurface 
\(\Gamma \subset \To^n\) with an orientation \(\nu\) has mean curvature in weak sense, if there exists a Borel function \(h \in L^1(\Gamma)\) such that for every \(f \in C^\infty(\To^n;\mathbb R^n)\) 
\[
\int_\Gamma \div_\tau f \ \d \H^{n-1} =  \int_\Gamma h \langle f,\nu \rangle \ \d \H^{n-1}.
\]
In such a case \(h\) is called a \emph{weak} or \emph{distributional mean curvature} of \(\Gamma\). Moreover, if \(h\) is constant, then \(\Gamma\) is called \emph{stationary}. By writing \(\Gamma\) locally as a graph surface in an orthonormal coordinates and applying the elliptic regularity results \cite[Proposition 7.56]{AFP} and \cite[Theorem 9.19]{GT} we obtain the following well-known result.
\begin{lemma}
\label{Lwcrit}
\ \\
Let \(\Gamma \subset \To^n\) be a compact \(C^{1,\alpha}\)-hypersurface with \(0<\alpha < 1\). Then \(\Gamma\) is smooth and critical if and only if it is stationary.
\end{lemma}

\subsection*{Vector fields, tensors and covariant derivatives on hypersurfaces}
Since now for every \(\varphi \in C^\infty(\Sigma)\) there is a smooth extension \(\tilde \varphi\) to some open neighborhood of \(\Sigma\) and tangential differential is independent of the way an ambient function is extended beyond \(\Sigma\), we may
define a tangential gradient \(\nabla_\tau \varphi (x)\) of \(\varphi\) at \(x\) by setting \(\nabla_\tau \varphi (x) = \nabla_\tau \tilde \varphi (x)\). Let \(T_x\Sigma\) be the \emph{tangent space} of \(\Sigma\) at \(x\). Then for every \(v \in T_x \Sigma\)
\[
v(\varphi) =  \langle z_v, \nabla_\tau \varphi (x) \rangle
\]
with some unique \(z_v \in G_x\Sigma\). Thus the geometric tangent space \(G_x\Sigma\) can be canonically
identified with the tangent space \(T_x\Sigma\) and, from now on, we just use the notation \(T_x\Sigma\). The tangent bundle \(T\Sigma\) 
is then the disjoint union of the tangent spaces  \(T_x \Sigma\) equipped with the corresponding smooth structure.
Further, we may canonically identify the set of smooth vector fields \(\mathfrak X(\Sigma)\) on \(\Sigma\), that is, the smooth sections
\(\Sigma \rightarrow T\Sigma\), with the collection 
\[
\mathfrak X(\Sigma) = \{X \in C^\infty(\Sigma;\mathbb R^n) : X(x) \in T_x\Sigma \ \ \text{for every} \ \ x \in \Sigma\},
\] 
and for \(X \in \mathfrak X(\Sigma)\) its action on \(\varphi \in C^\infty(\Sigma)\) can be seen as
\(X \varphi = \langle X , \nabla_\tau \varphi \rangle\) on \(\Sigma\).
As usual, for \(X,Y \in \mathfrak X(\Sigma)\) the vector field \(XY\) is determined by the rule \(XY\varphi = X(Y\varphi)\).

The Riemannian metric \(g\) on \(\Sigma\) is the naturally induced flat metric. Keeping the previous identifications in mind this is just the restriction of the standard Euclidean inner product to the hyperspace \(T_x\Sigma\). The usual flat and sharp operations induced by \(g\) for smooth vector and \emph{covector} fields on \(\Sigma\) are denoted by \(\flat\) and \(\sharp\) correspondingly. Then for \(\varphi \in C^\infty(\Sigma)\) the gradient vector field is
\(\text{grad} \ \varphi = d \varphi^\sharp = \nabla_\tau \varphi\).

Slightly abusing the notations we define for a vector field \(X \in \mathfrak X(\Sigma)\) the (geometric) tangential differential \(\D_\tau X (x)\) at \(x\)
as the linear map \(\mathbb R^n \rightarrow \mathbb R^n\) (or equivalently \(T_x \Sigma \rightarrow T_x \Sigma\))
 by setting
\[
\D_\tau X (x) = P_\Sigma(x) \D_\tau \tilde X (x),
\]
where \(\tilde X\) is any smooth extension of \(X\) beyond \(\Sigma\). The tangential divergence  \(\div_\tau X (x)\) of \(X\) at \(x\) is the trace of previous operator and \(\div_\tau X (x) = \div_\tau \tilde X (x)\). Now 
the mappings  \(\Sigma \rightarrow \mathcal L (\mathbb R^n;\mathbb R^n)\), \(x \mapsto \D_\tau X (x)\) 
and \(\Sigma \rightarrow \mathbb R\), \(x \mapsto \div_\tau X(x)\), are smooth. Again, for any \(\varphi \in C^\infty(\Sigma)\) the \emph{tangential Hessian} is given by \(\D_\tau^2 \varphi = \D_\tau \nabla_\tau \varphi\), which is a symmetric operator.
Further, the \emph{Laplace-Beltrami operator} or the \emph{tangential Laplacian} \(\Delta_\tau\) (of \(\Sigma\)) acting on \(\varphi\) can be seen as \(\Delta_\tau \varphi = \tr \ \D_\tau^2 \varphi = \div_\tau \nabla_\tau \varphi\).

The compatible Riemannian connection on \(\Sigma\) is the \emph{tangential connection} \(\nabla^\top : \mathfrak X(\Sigma) \times  \mathfrak X(\Sigma) \rightarrow \mathfrak X(\Sigma)\) given by the rule
\[
\nabla^\top_X Y = (\D_\tau Y) X.
\]
Recall that this is a symmetric connection, i.e.,
\(\nabla^\top_X Y-\nabla^\top_Y X = [X,Y]\), where  \([X,Y] =XY-YX \) is the corresponding commutator.

As usual, for every \(m \in \mathbb N \cup \{0\}\) and \(x \in \Sigma\) we denote the space of \(m\)-multilinear mappings or \(m\)-\emph{covariant tensors}
\(T_x\Sigma \times \cdots \times T_x\Sigma \rightarrow \mathbb R\) by \(T^m_0 (T_x \Sigma)\). Recall the special cases 
\(T^1_0(T_x\Sigma) = (T_x\Sigma)^*\) and \(T^0_0(T_x\Sigma)=\mathbb R\).
If  \(m>0\), for given \(L,G \in T^m_0 (T_x \Sigma)\) the inner product is given by
\[
\langle L,G\rangle = \sum_{i_1 \ldots i_m} L(v_{i_1},\ldots,v_{i_m}) G(v_{i_1},\ldots,v_{i_m}),
\]
where \(v_1,\ldots,v_{n-1}\) is any orthonormal basis of \(T_x\Sigma\). Thus the corresponding tensor norm for
\(L \in T^m_0 (T_x \Sigma)\) is \(|L| = \langle L,L \rangle^\frac{1}{2}\). If \(L\in T^2_0 (T_x \Sigma)\), then with help of 
 any orthonormal basis \(v_1,\ldots,v_{n-1}\) of \(T_x\Sigma\) 
 we may define the trace of 
\(L\) by setting 
\[
\tr \ L = \sum_i L(v_i,v_i).
\]

The \(m\)-\emph{covariant tensor bundle} \(T^m_0(\Sigma)\)
is defined by similar means as \(T\Sigma\). A \emph{covariant} \(m\)-\emph{tensor field} or section \(T:\Sigma \rightarrow T^m_0(\Sigma)\) is smooth
exactly when \(T(X_1,\ldots,X_m) \in C^\infty(\Sigma)\) for every \(X_1,\ldots,X_m \in \mathfrak X(\Sigma)\). We denote them by \(\mathscr T^m(\Sigma)\).
Recall that \(\mathscr T^1(\Sigma)\) is the collection of the smooth \emph{covector} fields or 1-\emph{forms} and  
\(\mathscr T^0(\Sigma) = C^\infty(\Sigma)\).

Now for any (smooth) covariant \(m\)-tensor field \(T \in \mathscr T^m(\Sigma)\) the covariant derivative of \(T\), denoted by 
\(\nabla_\co T\), is defined as the element of \(\mathscr T^{m+1}(\Sigma)\) for which
\[
\nabla_\co T (X_1,\ldots,X_m,X_{m+1}) = X_{m+1} T (X_1,\ldots,X_m) - \sum_{k=1}^m T(X_1,\ldots,\nabla_{X_{m+1}}^\top X_k,\ldots,X_m).
\]
For \(\varphi \in C^\infty(\Sigma)\) this simply means that \(\nabla_\co \varphi = d \varphi\) and again for \(T \in \mathscr T^m(\Sigma)\) the \(k+1\)-th 
covariant derivative is defined recursively by setting \(\nabla_\co^{k+1} T = \nabla_\co(\nabla_\co^k T)\). It is straightforward to compute that
for  \(\varphi \in C^\infty(\Sigma)\) the \emph{covariant Hessian} \(\nabla_\co^2 \varphi\) is now the symmetric 2-tensor field obtained by
\[
\nabla_\co^2 \varphi (X,Y)= \langle X, \D^2_\tau \varphi Y\rangle \ \ \text{for every} \ \  X,Y \in \mathfrak X(\Sigma),
\]
 \(\Delta_\co \varphi := \tr \ \nabla_\co^2\varphi = \Delta_\tau \varphi\) and \(|\nabla_\co^2 \varphi| = |\D^2_\tau \varphi|\), where \( |\D^2_\tau \varphi|\) is the standard Frobenius norm of \(\D^2_\tau \varphi\).

For any \(T \in \mathscr T^m(\Sigma)\) the \(C^k\)-norm on \(\Sigma\) is given in the obvious way
\[
\|T\|_{C^k(\Sigma)} = \sum_{i=0}^m \sup_\Sigma|\nabla_\co^i T|
\] 
and further we set \(\|X\|_{C^k(\Sigma)} = \|X^\flat\|_{C^k(\Sigma)}\) for every \(X \in \mathfrak X (\Sigma)\).

\subsection*{Sobolev and Hölder spaces on hypersurfaces}

Recall that the \emph{Riemannian measure} on \(\Sigma\) induced by the flat metric \(g\)  is the restriction of the \(n-1\)-dimensional Hausdorff measure \(\H^{n-1}\) to \(\Sigma\). Then the corresponding \(p\)-Lebesgue space is \(L^p(\Sigma, \H^{n-1})\) for any \(1\leq p \leq \infty\). 
For every smooth covariant tensor field \(T\) on \(\Sigma\) the \(L^p\)-norm is now given by
\[
\|T\|_{L^p(\Sigma)} = \left(\int_\Sigma |T|^p \ \d \H^{n-1} \right)^\frac{1}{p}.
\]
For a \(\varphi \in C^\infty(\Sigma)\) the Sobolev \(W^{k,p}\)-norm with \(k \in \mathbb N \cup \{0\}\) and \(1\leq p \leq \infty\) is given
as usual
\[
\|\varphi\|_{W^{k,p}(\Sigma)} = \left(\sum_{j=0}^k \|\nabla_\co^j \varphi \|_{L^p(\Sigma)}^p\right)^\frac{1}{p}.
\]
Here \(\nabla_\co^0\varphi = \varphi\). For any \(1\leq p < \infty\) the space \(W^{k,p}(\Sigma)\) is now the norm completion of  \(C^\infty(\Sigma)\),
where any of its element \(\varphi\) is considered as the \(k+1\)-tuple \((\varphi,\nabla_\co^1\varphi,\ldots,\nabla_\co^k\varphi)\). 
Again, we use the conventional notation \(H^k(\Sigma)\) for a Hilbert space
\(W^{k,2}(\Sigma)\), where the inner product is given in the obvious way 
\[
\langle \varphi_1,\varphi_2\rangle_{H^k} = \sum_{j=0}^k \int_\Sigma \langle\nabla_\co^j \varphi_1,\nabla_\co^j\varphi_2\rangle \ \d\H^{n-1}.
\]

We have the standard Sobolev interpolation for \(L^p\)-norms of covariant derivatives of smooth maps, see \cite[Theorem 3.70]{Au}. From now on, we denote the space of \(C^\infty(\Sigma)\)-maps with vanishing integrals by
\(\tilde C^\infty(\Sigma)\).
 
\begin{lemma}[\textbf{Basic interpolation}]
\label{Int1}
\ \\
Let \(1\leq r,q < \infty\), \(1\leq p \leq \infty\) and integers \(0 \leq j < m\) be selected  
such that
\begin{equation}
\label{Int1a}
\frac{1}{p} = \frac{j}{n-1} + \left(\frac{1}{r} - \frac{m}{n-1}\right)\alpha + \frac{1-\alpha}{q},  
\end{equation}
where \(\frac{j}{m} \leq \alpha \leq 1\) and the condition \(r= \frac{n-1}{m-j} \neq 1 = \alpha\) does not hold.
Then there exists a constant \(C\) depending on the previous numbers and \(\Sigma\) such that for every \(\varphi \in C^\infty(\Sigma)\)
\begin{equation}
\label{Int1b}
\|\nabla_{\co}^j \varphi\|_{L^p(\Sigma)} \leq C \left( \|\nabla_{\co}^m \varphi\|_{L^r(\Sigma)}^\alpha \| \varphi\|_{L^q(\Sigma)}^{1-\alpha} + \beta \|\varphi\|_{L^1(\Sigma)}\right).
\end{equation}
Here \(\beta = 0\), if  (i) \(j\geq1\) or (ii) \(\Sigma\) is connected and \(\varphi \in \tilde C^\infty(\Sigma)\). Otherwise, \(\beta = 1\).
\end{lemma}

\begin{rem}
\label{H3Remark}
\ \\
It follows from the previous theorem  that in the case \(2\leq n \leq 4\) there is a
constant \(C\) depending on \(\Sigma\) such that \(\|\varphi\|_{L^p(\Sigma)} \leq C \|\varphi\|_{H^1(\Sigma)}\) and \(\|\varphi\|_{C^1(\Sigma)},\|\varphi\|_{W^{2,p}(\Sigma)} \leq C \|\varphi\|_{H^3(\Sigma)}\) for every \(1\leq p\leq 6\) and \(\varphi \in C^\infty(\Sigma)\).
\end{rem}

We also recall the following interpolation inequality for covariant tensor fields, see \cite[Theorem 12.1]{Ha}.

\begin{lemma}[\textbf{Second order tensor interpolation}]
\label{Int2}
\ \\
Let $N$ be the number of connected components of $\Sigma$ and suppose that 
\begin{equation}
\label{Int2a}
\frac{1}{p} = \frac{1}{q}+ \frac{1}{r}  
\end{equation}
for \(1\leq p,q,r \leq \infty\).
Then for every smooth covariant \(k\)-tensor field \(T\) on \(\Sigma\) it holds
\begin{equation}
\label{Int2b}
\|\nabla_{\co} T\|_{L^{2p}(\Sigma)}^2 \leq N (2p-2+n-1) \|\nabla_{\co}^2 T\|_{L^r(\Sigma)} \| T\|_{L^q(\Sigma)}.
\end{equation}
\end{lemma}

Moreover, we need the following estimates, see \cite[Lemma 2.3 and Remark 2.4]{FJM}.
\begin{lemma}
\label{Llaplace}
\ \\
There are
constants \(C_n\), depending only on \(n\), and \(C_\Sigma\), depending on \(\Sigma\), such that for every
\(\varphi \in C^\infty(\Sigma)\)
\begin{align}
\label{llaplace1}
\|\nabla^2_{\co} \varphi\|_{L^2(\Sigma)}^2 &\leq \|\Delta_{\co} \varphi\|_{L^2(\Sigma)}^2 + C_n \int_\Sigma 
|B|^2|\nabla_{\co} \varphi|^2 \ \d\mathcal H^{n-1} \\
\label{llaplace2}
\|\nabla^3_{\co} \varphi\|_{L^2(\Sigma)}^2 &\leq \|\nabla_\co(\Delta_{\co} \varphi)\|_{L^2(\Sigma)}^2 + C_\Sigma \|\varphi\|_{H^2(\Sigma)}^2.
\end{align}
\end{lemma}

For a continuous map \(f:\Sigma \rightarrow \mathbb R\) and \(0 <\alpha < 1\) the \(C^\alpha(\Sigma)\)-Hölder semi-norm
is given by 
\[
[f]_{C^\alpha(\Sigma)} = \max_{i=1,\ldots,N} \sup_{\substack{ x,y \in \Sigma_i \\ x \neq y}} \frac{|f(x)-f(y)|}{d_g (x,y)^\alpha},
\]
where \(d_g\) is the length metric induced by \(g\). For every \(\varphi \in C^{\infty}(\Sigma)\) we define 
\(C^{1,\alpha}(\Sigma)\)-norm by setting
\[
\|f\|_{C^{1,\alpha}(\Sigma)} = \|\varphi\|_{C^1(\Sigma)} + \sup_{\substack{X \in \mathfrak X(\Sigma) \\ \|X\|_{C^1(\Sigma)} \leq 1}} [\nabla_\co \varphi (X)]_{C^\alpha(\Sigma)} 
\]
and set the space \(C^{1,\alpha}(\Sigma)\) to be the norm completion of the set of \(C^\infty(\Sigma)\)-maps with finite \(C^{1,\alpha}(\Sigma)\)-norm. 
Then  \(C^{1,\alpha}(\Sigma)\)  is the space of continuous maps on \(\Sigma\) with \(C^{1,\alpha}\)-extension to some open neighborhood of \(\Sigma\).
Note that there is several equivalent ways to define \(C^{1,\alpha}(\Sigma)\)-norm.

The higher order Hölder spaces are defined similarly, but we do not need them. By using the Rellich-Kondarchov theorem, see \cite[Thm 2.9]{He}, and Lemma \ref{Int1}
one obtains the following embedding result.
\begin{lemma}
\label{Emb}
\ \\
Suppose that \(k \geq 3\) is an integer and \(1< p \) is a real number.
If for given exponent \(0<\alpha<1\) the condition
\[
\alpha <  2 - \frac{n-1}{p}
\]
is true, then the embedding \(W^{3,p}(\Sigma) \subset C^{1,\alpha}(\Sigma)\) is compact. 
\end{lemma}

In particular, Lemma \ref{Emb} says that for \(n \leq 4\) and \(0<\alpha<\frac{1}{2}\) the embedding \(H^3(\Sigma) \subset C^{1,\alpha}(\Sigma)\) is compact.

\subsection*{Smooth sets}

An open set \(E \subset \To^n\), with \({\mathrm{int}(\overline E)} =  E\), is called a \emph{smooth set}, if the boundary \(\partial E\) is a smooth 
hypersurface. Then \(\partial E\) is always a smooth and compact hypersurfaces in \(\To^n\) so the previous 
results can be applied on \(\partial E\). For the boundary \(\partial E\) we always use the natural inside-out orientation denoted by \(\nu_E\). The \emph{classical divergence theorem} takes now the following form in \(\To^n\). For any Lipschitz function \(f:\To^n \rightarrow \mathbb R^n\)  
\[
\int_E \div \ f \ \d \H^n = \int_{\partial E} \langle f, \nu_E \rangle \ \d \H^{n-1}.
\]
Further, we denote by \(B_E\) the second fundamental form on \(\partial E\) associated with \(\nu_E\) and by \(H_{E}\) the corresponding boundary mean curvature. Then \(E\) is said to be 
critical, if \(\partial E\) is critical, i.e, \(H_E\) is constant.
 We use also the shorthand notation \(|E|\) for the volume \(\H^n(E)\).
 We recall that there exists 
a \emph{regular neighborhood} of \(\partial E\) say \(U_E\) such that the \emph{signed distance function} \(\bar d_E : \To^n \rightarrow \mathbb R\)
\[
\bar d_E(x) = 
\begin{cases}
\dist(x,\partial E), &x \in \To^n\setminus E \\
-\dist(x,\partial E), &x \in  E 
\end{cases}
\]
and the projection mapping \(\pi_{\partial E}\) onto \(\partial E\) are smooth on \(\bar U_E\) (in particular the latter is well-defined). Again, we may write \(\nu_E = \nabla \bar d_E\) and \(B_E= \D^2  \bar d_E\) on \(\partial E\).
The \(C^{k,\alpha}\)-sets are defined similarly.

For any \(C^k\)-sets \(E\) and \(E'\) in \(\To^n\) the
\(C^k\)-“distance"  is given by
\[
\|E,E'\|_{C^k} = \inf \{\|\Phi - \id\|_{C^k(\To^n;\To^n)} : \Phi \ \text{is a \(C^k\)-diffeomorphism with} \ \Phi(E)=E' \}.
\]

For a smooth set \(E\) and an open set \(E'\) we denote \(E'=E_\psi\) for some \(\psi \in C (\partial E)\), if we may write \(\partial E'\) as a \emph{graph of} \(\psi\) \emph{in normal direction over} \(\partial E\), that is,
\begin{equation}
\partial E' = \{x + \psi(x)\nu_E(x) :x \in \partial E\}
\end{equation}
and \(x+ (s+\psi(x))\nu_E(x) \in \To^n \setminus E'\) with a small positive \(s\) for every \(x \in \partial E\). 
Now small \(C^1\)-distance between \(E\) and a \(C^{k,\alpha}\)-set \(E'\) is equivalent to
the graph representation \(E'=E_\psi\) with a \(C^{k,\alpha}(\partial E)\)-map \(\psi\) having a small \(C^1\)-norm.

\begin{lemma}
\label{Lc1equivalence}
\ \\
Let \(E \subset \To^n\) be a smooth set. There exist positive constants \( C\geq 1\) and \(\delta\) depending on \(E\)
such that the set  \(\{y \in \To^n: |\bar d_E(y)| \leq \delta\}\) belongs to a regular neighborhood of \(\partial E\) and 
the following hold for any \(k \in \mathbb N \cup \{\infty\}\) and \(0\leq \alpha < 1\).
\begin{itemize}
\item[(i)] For any \(\psi \in C^{k,\alpha}(\partial E)\) with \(\|\psi\|_{C^1(\partial E)} \leq \delta\) the set \(E_\psi\) is defined as 
a \(C^{k,\alpha}\)-set, the map \(\Phi_\psi: \partial E \rightarrow \partial E_\psi\), given by
\[
\Phi_\psi(x) = x + \psi(x)\nu_E(x),
\]
 is \(C^{k,\alpha}\)-diffeomorphism and
\(
\|E,E_\psi\|_{C^1} \leq C \|\psi\|_{C^1(\partial E)}.
\)
\item[(ii)]
If \(E' \subset \To^n\) is a \(C^{k,\alpha}\)-set with \(\|E,E'\|_{C^1} \leq \delta\), then there is a unique map \(\psi \in C^{k,\alpha}(\partial E)\) for which
\(E' = E_\psi\) and
\(
 \|\psi\|_{C^1(\partial E)} \leq C \|E,E'\|_{C^1}.
\)
Moreover, if \((E'_t)_t\) is a smooth flow (see Definition \ref{smoothflow}) in \(\To^n\), with \(\|E,E'_t\|_{C^1} \leq \delta\), then the maps \(\psi_t\), with \(E'_t = E_{\psi_t}\), are smoothly parametrized on \(\partial E\).
\end{itemize}
\end{lemma}

When for \(\psi \in C^\infty(\partial E)\) its \(C^1\)-norm is small enough, we may control \(B_{E_\psi}\) and represent \(H_{E_\psi}\) via \(\Phi_\psi\) on \(\partial E\) in the following way. 

\begin{lemma}
\label{LMC}
\ \\
Let \(E \subset \To^n\) be a smooth set. There are
constants \(\delta=\delta(E)\) and \(C=C(E)\) and smooth maps
\begin{align*}
A&:\partial E  \times [-\delta,\delta]\times[-\delta,\delta]^n \rightarrow T^2_0(\partial E),\\
Z&:\partial E \times [-\delta,\delta]\times[-\delta,\delta]^n \rightarrow T(\partial E) \ \ \ \text{and}\\
P&:\partial E \times [-\delta,\delta]\times[-\delta,\delta]^n \rightarrow \mathbb R
\end{align*}
depending on \(E\)  such that \(A(\ \cdot \ ,t,z)\) and \(Z(\ \cdot \ ,t,z)\)  are sections for every pair \((t,z) \in  [-\delta,\delta]\times[-\delta,\delta]^n \), \(A(\ \cdot \ ,0,0) = 0\) and the following hold.
If \(\psi \in C^\infty(\partial E)\) and \(\|\psi\|_{C^1(\partial E)} \leq \delta\), then 
\begin{align}
\notag
H_{E_\psi} \circ \Phi_\psi
&= - \Delta_\tau \psi + \langle A (\ \cdot \ ,\psi,\nabla_\tau \psi ),\nabla^2_\co \psi\rangle \\
\label{lMC1}
& +  \nabla_\co \psi\left(Z(\ \cdot \ ,\psi,\nabla_\tau \psi )\right) +\psi P (\ \cdot \ ,\psi,\nabla_\tau \psi ) + H_E  \ \ \ \text{and} \\
\label{lMC2}
|B_{E_\psi} \circ \Phi_\psi| &\leq C(1+|\D_\tau^2\psi|)
\end{align}
on \(\partial E\).
\end{lemma}

Moreover, with the same \(\delta\) and \(C\) we may assume that the following holds. If \(\psi \in C^\infty(\partial E)\) with \(\|\psi\|_{C^1(\partial E)} \leq \delta\), then for every \(1\leq p < \infty\), \(h \in L^p(\partial E_\psi)\) and \(\varphi \in C^\infty(\partial E_\psi)\)
\begin{align}
\label{unifcontrol1a}
C^{-1} \| h \circ \Phi_\psi\|_{L^p(\partial E)} &\leq  \|h\|_{L^p(\partial E_\psi)} \leq C \| h \circ \Phi_\psi\|_{L^p(\partial E)} \ \ \ \text{and} \\
\label{unifcontrol1b}
C^{-1} \|\nabla_\tau(\varphi \circ \Phi_\psi)\|_{L^p(\partial E)} &\leq  \|\nabla_\tau \varphi\|_{L^p(\partial E_\psi)} \leq C \|\nabla_\tau(\varphi \circ \Phi_\psi)\|_{L^p(\partial E)}.
\end{align}

Finally we need the following uniform estimates for low dimensions. The first one says that the first estimate of Remark \ref{H3Remark}  holds with a uniform constant 
when slightly varying a reference boundary in \(C^1\)-sense.
 
\begin{lemma}
\label{LC1unifcontrol}
\ \\
Let \(E \subset \To^n\), \(n=3,4\), be a smooth set.
There exist positive constants \(\delta\) and \(C\) depending on \(E\) such that  
if \(\psi \in C^\infty(\partial E)\) with \(\|\psi\|_{C^1(\partial E)} \leq \delta\), then for every \(1\leq p \leq 6\) and \(\varphi \in C^\infty(\partial E_\psi)\)
\begin{equation*}
\|\varphi\|_{L^p(\partial E_\psi)} \leq C \|\varphi\|_{H^1(\partial E_\psi)}.
\end{equation*}
\end{lemma}

The second one says that  we can control uniformly \(L^p\)-norm of \(\nabla_\tau \varphi\) (up to \(p\leq 6\)) by the \(L^2\)-norm of \(\Delta_\tau \varphi\) and the \(H^1\)-norm of \(\varphi\) when slightly varying a reference boundary in \(H^3\)-sense.

\begin{lemma}
\label{LH3unifcontrol}
\ \\
Let \(E \subset \To^n\), \(n=3,4\), be a smooth set. There are
constants \(\delta\) and \(C\) depending on \(E\) such that if \(\psi \in C^\infty(\partial E)\) and \(\|\psi\|_{C^1(\partial E)} \leq \delta\), then 
for every \(1\leq p \leq 6\) and \(\varphi \in C^\infty(\partial E_\psi)\)
\begin{align*}
\|\nabla_\tau \varphi\|_{L^p(\partial E_\psi)} \leq C \left(\|\Delta_\tau \varphi\|_{L^2(\partial E_\psi)} +  \|\varphi\|_{H^1(\partial E_\psi)}\right).
\end{align*}
\end{lemma}

Lemma \ref{LC1unifcontrol} and Lemma \ref{LH3unifcontrol} hold also in the case \(n>4\), if we replace the upper bound \(6\) with smaller number depending on \(n\). This number converges to \(2\) by above as \(n\) tends to infinity.
We will prove these results except Lemma \ref{Lc1equivalence} in Appendix.

\subsection*{Volume preserving mean curvature flow and strictly stable sets}

Let us first give the formal definition of smooth flow in this setting.
\begin{defi}[\textbf{Smooth flow}]
\label{smoothflow}
\ \\
The parametrized family \((E_t)_{t \in [0,T[}\) of smooth sets in \(\To^n\) with \(0<T\leq\infty\) is a smooth flow with an \emph{initial set} or \emph{initial datum}  \(E_0\), if there 
exists a smooth map \(\Phi: \To^n \times [0,T[ \ \rightarrow \To^n\) such that \(\Phi_0 = \id_{\To^n}\), \(\Phi_t := \Phi( \ \cdot, t)\) is a (smooth) diffeomorphism  and 
\(E_t = \Phi_t(E_0)\) for every \(t \in [0, T[\). Again, for every \(0\leq t < T\), the (outer) normal velocity of the flow on \(\partial E_t\) at the time \(t\) is defined by setting
\[
V_t = \langle \partial_s \Phi_{s+t} \Big |_{s=0} \circ \Phi_t^{-1}, \nu_{E_t} \rangle.
\]   
\end{defi}

The normal velocity \(V_0\) on \(\partial E_0\) is called an \emph{initial velocity}. If we do not emphasize the time interval \([0,T[\), we write \((E_t)_t\) for short. Also, for the unit normal field \(\nu_{E_t}\) the corresponding second fundamental form \(B_{E_t}\) and the boundary mean curvature  \(H_{E_t}\) we use the shorthand notations 
\(\nu_t\), \(B_t\) and \(H_t\), when there 
is no possibility of confusion. In the previous definition \(\Phi\) is called a \emph{smoothly parametrized family of diffeomorphisms} and we may suggestively denote it by \((\Phi_t)_{t \in [0,T[}\) or \((\Phi_t)_t\). The normal velocity \(V_t\) on \(\partial E_t\) does not depend on the choice of the parametrization \(\Phi\).
Recall the \emph{first variation of volume} under the flow 
\begin{equation}
\label{1stvarV}
\frac{\d}{\d t} |E_t| = \int_{\partial E_t} V_t \ \d \H^{n-1}
\end{equation}
and again  the \emph{first variation of perimeter}
\begin{equation}
\label{1stvarP}
\frac{\d}{\d t} |\partial E_t| = \int_{\partial E_t} V_t H_t \ \d \H^{n-1}.
\end{equation}

We say that \((E_t)_t\) is \emph{volume preserving}, if \(|E_t|\) is a constant function in time. According to \eqref{1stvarV} this is possible exactly when \(V_t\) has a vanishing integral over \(\partial E_t\) for every \(t\). Conversely, one may show that if \(E\) is a smooth set in \(\To^n\) and 
\(\varphi \in \tilde C^\infty(\partial E)\), there is a smooth volume preserving flow \((E_t)_t\) starting from \(E\) such that
 initial velocity on \(\partial E\) is \(\varphi\).
Then, by using a simple approximation argument, it follows from \eqref{1stvarP} that a smooth set \(E\) is critical if and only if for every smooth volume preserving flow \((E_t)_t\) starting from \(E\)
\[
\frac{\d}{\d t} |\partial E_t| \Big |_{t=0} = 0. 
\]

Further, if in the previous setting the flow admits a parametrization \(\Phi\) autonomous with respect to time, i.e., 
\(\partial_t \Phi = X(\Phi)\) with a \(X \in C^\infty(\To^n;\mathbb R^n)\), then the \emph{second variation of perimeter} at \(t=0\) is
\[
\frac{\d^2}{\d t^2} |\partial E_t| \Big |_{t=0} = \int_{\partial E} |\nabla_\tau V_0|^2 - |B_E|^2V_0^2 \ \d \H^{n-1},
\] 
see \cite[Remark 3.3]{AFM}. This motivates us to define for every smooth \(E\) in \(\To^n\) the quadratic form \(\partial^2 P(E) : \tilde H^1 (\partial E) \rightarrow \mathbb R\), where \(\tilde H^1 (\partial E)\) is the space of  \(H^1 (\partial E)\)-maps with vanishing integral over \(\partial E\), by setting first for every \(\varphi \in \tilde C^\infty(\partial E)\) 
\[
\partial^2 P(E) [\varphi] = \int_{\partial E} |\nabla_\tau \varphi|^2 - |B_E|^2\varphi^2 \ \d \H^{n-1}
\]
and then extending \(\partial^2 P(E)\) to  \(\tilde H^1 (\partial E)\) in the obvious way. We say that a smooth and critical set \(E\) in \(\To^n\) is stable, if 
\(\partial^2 P(E)(\varphi)\geq 0\) for every \(\varphi \in \tilde H^1(\partial E)\).

For every smooth set \(E\) in \(\To^n\) the space of \emph{infinitesimal translations} on \(\partial E\) is given by
\[
T(\partial E) = \{ \langle \nu_E, p \rangle : p \in \mathbb R^n\}. 
\]  
Notice that \(T(\partial E) \subset \tilde C^\infty(\partial E)\).
These maps correspond to the initial velocities of the linear translations along a vector. It follows from \cite[Theorem 3.1]{AFM} that 
if \(E_t = E +tp\), then the second variation at \(t=0\) of perimeter under this translation is \(\partial^2 P(E)[\langle \nu_E, p \rangle]\).
Since translations do not change perimeter, then \(\partial^2 P(E)[\langle \nu_E, p \rangle] = 0\). Thus
\(\partial^2 P(\partial E)\) is always zero on \(T(\partial E)\), which is also easy to compute directly by using the definition.
This leads us to define the strictly stable sets in the following way, see \cite{AFM}.

\begin{defi}[\textbf{Strictly stable set}]
\label{StricSet}
\ \\ 
A smooth and critical set \(F\) in \(\To^n\) is strictly stable, if 
\(\partial^2 P(E) [\varphi]>0\) for every nonzero map \(\varphi \in \tilde H^1 (\partial F) \cap T^{\perp,L^2}(\partial F)\), where
\( T^{\perp,L^2}(\partial F)\) is the orthogonal complement of \(T(\partial F)\) in \(L^2(\partial F)\). 
\end{defi}
In the previous definition, we may replace the condition \(\partial^2 P(E) [\varphi]>0\) for every nonzero 
map \(\varphi \in \tilde H^1 (\partial F) \cap T^{\perp,L^2}(\partial F)\) with \(\partial^2 P(E) [\varphi]>0\) for every map \(\varphi \in \tilde H^1 (\partial F) \setminus T(\partial F)\) due to the fact that \(\Delta_\tau \langle \nu_F, p \rangle = -|B_F|^2\langle \nu_F, p \rangle\) for every \(p \in \mathbb R^n\).
As mentioned already in  Introduction such sets are always isolated local perimeter minimizers. For a strictly stable set \(F\) any critical set with the same volume being close enough to \(F\) in \(H^3\)-sense 
is a translate of \(F\).
\begin{lemma}
\label{Ltransl0}
\ \\
Let \(F \subset \To^n\), \(n=3,4\), be a strictly stable set. There exists a positive \(\delta = \delta (F)\) such that if 
\(F_\psi\) is critical, \(|F_\psi| = |F|\) and \(\|\psi\|_{H^3(\partial F)} \leq \delta_1 \) for \(\psi \in C^\infty(\partial F)\), then
\(F_\psi\) is a translate of \(F\).
\end{lemma}
This result follows from the proof of \cite[Theorem 3.9]{AFM} (for clarity see the proof of \cite[Proposition 2.7]{AFJM} although it concerns only the case \(n=3\)) and Remark \ref{H3Remark}.
Moreover, being near to a strictly stable set in \(H^3\)-sense implies that the quadratic form controls \(H^1\)-norm for the \(\tilde C^\infty\)-functions 
orthogonal to the infinitesimal translations in \(L^2\)-sense, see
\cite[Lemma 2.6]{AFJM}.
\begin{lemma}
\label{Lsvar}
\ \\
Let \(F \subset \To^n\), \(n=3,4\), be a strictly stable set. There exist \(\sigma_1 = \sigma_1 (F)\) and \(\delta = \delta (F)\) such that if 
\(\|\psi\|_{H^3(\partial F)} \leq \delta_1 \) for \(\psi \in C^\infty(\partial F)\), then for every \(\varphi \in  \tilde C^\infty (\partial F_\psi) \cap  T^{\perp, L^2}(\partial F_\psi)\)
\begin{equation*}
 \sigma_1 \|\varphi\|^2_{H^1(\partial F_\psi)} \leq \partial^2 P(F_\psi)[\varphi].
\end{equation*}
\end{lemma}
The authors prove the previous result in the case \(n=3\) and \(H^3\)-closeness being replaced by \(W^{2,p}\)-closeness with any \(p>2\), 
but clearly the same arguments go through for any \(n\geq 2\), if we use the condition \(p > \max\{n-1,2\}\) instead of \(p>2\). 
Hence the result follows from Remark \ref{H3Remark}.

As already presented in Introduction, a volume preserving mean curvature flow \((E_t)_t\)  in \(\To^n\) is a smooth flow obeying the rule
\[
V_t = \bar H_t - H_t,
\]
where \(\bar H_t\) is the integral average of \(H_t\) over \(\partial E_t\). Then by \eqref{1stvarV} and \eqref{1stvarP} we see that
\[
\frac{\d}{\d t} |E_t| = 0 \ \ \ \text{and} \ \ \ \frac{\d}{\d t} |\partial E_t| = - \int_{E_t} (\bar H_t - H_t)^2 \ \d \H^{n-1} \leq 0,
\]
so the flow is volume preserving and  decreases boundary area. The existence and uniqueness of such a flow for a given smooth initial datum is well-known. This
is usually known as \emph{short time existence}. Also, the maximal lifetime of the flow is bounded by below when slightly varying the initial set in the \(C^{1,\alpha}\)-topology with \(\alpha>0\).

\begin{thm}[\textbf{Short Time Existence}]
\label{Ex}
\ \\
Let \(E \subset \mathbb T^n\) be a smooth set and \(0<\alpha<1\). There are positive
constants \(\delta\) and \(T\) depending on \(E\) and \(\alpha\) such that
if \(E_0\) is a smooth set in \(\mathbb T^n\)  of the form \(\partial E_0 = E_{\psi_0}\),
where \(\psi_0 \in C^\infty(\partial E)\) and \(\|\psi_0\|_{C^{1,\alpha}(\partial E)} \leq \delta\), then there exists a unique volume preserving mean curvature flow  in 
\(\mathbb T^n\) with the initial datum \(E_0\) and the maximal lifetime of flow is at least \(T\).
\end{thm}

The above result has been proved for bounded smooth sets in \(\mathbb R^n\), see \cite[Main Theorem]{ES} and clearly it is similar to prove it in the setting of the 
flat torus \(\To^n\). 
 From now on, when there is no danger of confusion, we denote a maximal lifetime of a given volume preserving mean curvature flow \((E_t)_t\) by \(T^*\).

%%%%%%%%%%%%%%%%%%%%%%%%%%%%%%%%%%%%%%%%%%%%%%%%%%%%%%%%%%%%%%%%%%%%%%%%%%%%%%%%%%%%%%%%%%%%%%%%%%%%%%%%%%%%%%%%%%%%%%

\section{\(L^2\)-monotonicity}

In this section we prove a monotonicity result, which is the basis of our analysis. It states that if \((E_t)_t\) is a volume preserving mean curvature flow  with a smooth initial datum near enough to a strictly stable set \(F\) in \(H^3\)-sense, then  
it stays  near \(F\) in the \(H^3\)-sense for the whole lifespan of the flow, the initial velocity of flow decreases exponentially in \(L^2\)-sense and the quantity \(\|\nabla_{\tau}H_t\|_{L^2(\partial E_t)}^2+ C_0 \|\bar H_t - H_t\|_{L^2(\partial E_t)}^2\) is decreasing in time with sufficiently large \(C_0\). More precisely we have the following result.

\begin{thm}[\textbf{Monotonicity near strictly stable set}]
\label{M1}
\ \\
Let \(F \subset \mathbb T^n\) (\(n=3,4\)) be a strictly stable set. There are positive
constants \(C_0\) and \(\epsilon_0\) depending on \(F\) such that for every \(0<\epsilon \leq \epsilon_0\) there is a positive \(\gamma_\epsilon < \epsilon\)  such that if \((E_t)_t\) is a volume preserving mean curvature flow starting from a smooth set \(E_0 = F_{\psi_0}\),
where \(\psi_0 \in C^\infty(\partial F)\) and \(\|\psi_0\|_{H^3(\partial F)} \leq \gamma_\epsilon\), then for every \(t \in [0,T^*[\) (where \(T^*>0\) is the maximal lifetime of flow) we may write \( E_t = F_{\psi_t}\) with \(\psi_t \in C^\infty(\partial F)\), \(\|\psi_t\|_{H^3(\partial F)} \leq \epsilon\),
\begin{align}
\label{m1}
\|\bar H_t - H_t\|_{L^2(\partial E_t)}^2 &\leq \|\psi_0\|_{H^3(\partial F)} e^{-\sigma_1 t} \ \ \ \text{and} \\
\label{m2}
\frac{\d}{\d t}\left[ \|\nabla_{\tau}H_t\|_{L^2(\partial E_t)}^2+ C_0 \|\bar H_t - H_t\|_{L^2(\partial E_t)}^2 \right]
&\leq 0,
\end{align}
where \(\sigma_1\) is as in Lemma \ref{Lsvar}.
\end{thm}

\begin{rem}
\label{M2}
\ \\
As a byproduct of the proof of Theorem \ref{M1} we may replace zero with \( -\frac{1}{2}\| \Delta_{\tau} H_t \|_{L^2(\partial E_t)}^2\) on the right-hand side of \ref{m2}. However, we do not need that fact.
\end{rem}

Before proving Theorem \ref{M1} we introduce the following useful geometric quantity employed, for instance, in \cite{AFJM}.
For any open set \(E \subset \To^n\) we define the map \(D_E\) from the collection of the open subsets of \(\To^n\) to \([0,\infty[\)
by setting for any open \(E' \subset \To^n\)
\begin{equation}
D_E(E') := \int_{E'\Delta E} \dist_{\partial E} \ \d \H^n =  \int_{E'} \bar d_E \ \d \H^n -  \int_E \bar d_E \ \d \H^n.
\end{equation}
In case of any smooth set \(E \subset \To^n\) this concept of “weak distance" turns out to be very useful in terms of controlling the \(L^2(\partial E)\)-norm of the corresponding function of given \(C^1\)-graph in normal direction over \(\partial E\), when the corresponding \(C^1(\partial F)\)-norm is sufficiently small.
To observe this, choose \(\delta = \delta(E)\) so small that by Lemma \ref{Lc1equivalence} for any \(\psi \in C^1(\partial E)\) with \(\|\psi\|_{C^1(\partial E)} \leq \delta\) the set 
\(E_\psi\) is defined as a \(C^1\)-set. Moreover, we may assume that for every \(s \in [-\delta,\delta]\) the map 
\(\Phi_s = \id + s \nabla \bar d_E\) is defined as a smooth diffeomorphism from some tubular neighborhood of \(\partial E\) to its image.
With help of the coarea formula one may compute
\[
D_E(E_\psi) = \int_{\left[0,\|\psi\|_{L^\infty(\partial E)}\right]} s \left[ \int_{\{ x \in \partial E: \psi(x) > s \}} J_\tau \Phi_s \ \d \mathcal H^{n-1}
+\int_{\{ x \in \partial E : \psi(x) < -s \}} J_\tau \Phi_{-s} \ \d \mathcal H^{n-1}\right]\ \d s.
\]

\begin{comment}
\begin{align*}
D_E(E_\psi) 
&= \int_{E_\psi \Delta E}  \dist_{\partial E} \ \d \mathcal H^n = \int_{E_\psi \Delta E}  \frac{1}{2}\left|\nabla (\bar d_E)^2\right| \ \d \mathcal H^n \\
&= \frac{1}{2}\int_{\left[0,\|\psi\|_{L^\infty(\partial E)}^2\right]} \mathcal H^{n-1}\left( \{ y \in E_\psi \Delta E :  (\bar d_F (y))^2 = t \}\right) \ \d t \\
&= \int_{\left[0,\|\psi\|_{L^\infty(\partial E)}\right]} s  \mathcal H^{n-1}\left( \{ y \in E_\psi \Delta E :  |\bar d_E (y)| = s \}\right) \ \d s \\
&= \int_{\left[0,\|\psi\|_{L^\infty(\partial E)}\right]} s  \mathcal H^{n-1}\left( \{ y \in E_\psi \Delta E :  \bar d_E (y) = s \}\right) \ \d s \\
&+\int_{\left[0,\|\psi\|_{L^\infty(\partial E)}\right]} s \mathcal H^{n-1}\left( \{ y \in E_\psi \Delta E :  \bar d_E (y) = -s \}\right) \ \d s \\
&= \int_{\left[0,\|\psi\|_{L^\infty(\partial E)}\right]} s \mathcal H^{n-1}\left( \Phi_s \left(\{ x \in \partial E : \psi(x) > s \}\right)\right) \ \d s \\
&+\int_{\left[0,\|\psi\|_{L^\infty(\partial E)}\right]} s \mathcal H^{n-1}\left( \Phi_{-s} \left(\{ x \in \partial E : \psi(x) < -s \}\right)\right) \ \d s \\
&= \int_{\left[0,\|\psi\|_{L^\infty(\partial E)}\right]} s \left[ \int_{\{ x \in \partial E: \psi(x) > s \}} J_\tau \Phi_s \ \d \mathcal H^{n-1}
+\int_{\{ x \in \partial E : \psi(x) < -s \}} J_\tau \Phi_{-s} \ \d \mathcal H^{n-1}\right]\ \d s.
\end{align*}
\end{comment}

Since now \(J_\tau \Phi_s \rightarrow 1\) uniformly on \(\partial E\) as \(s \rightarrow 0\), then by decreasing \(\delta\), if necessary,  
it follows from the Cavalieri's principle that there exists \(C\geq 1\) depending on \(\delta\) such that for every \(\psi \in C^1(\partial E)\) with \(\|\psi\|_{C^1(\partial E)} \leq \delta\) 
\begin{equation}
\label{l2equivalence}
 C^{-1} \|\psi\|_{L^2(\partial E)}^2  \leq D_E(E_\psi) \leq C \|\psi\|_{L^2(\partial E)}^2. 
\end{equation}

Again, for any open subset \(E \subset \To^n\) the map \(D_E\) behaves well under any smooth flow.
If \((E_t)_t\) is a smooth flow in \(\To^n\) with a normal velocity \(V_t\), then \(D_E(E_t)\) is differentiable in time and
it is straightforward to calculate
\begin{equation}
\label{Dtime}
\frac{\d}{\d t} D_E(E_t) = \int_{\partial E_t}\bar d_E V_t \ \d \H^{n-1}.
\end{equation}

Now the importance of this quantity lies on the fact, that for any smooth critical set \(F\subset \To^n\) and \(\psi \in C^\infty(\partial F)\) with sufficiently small \(C^1\)-norm, \(\|\nabla_\tau H_{F_\psi}\|_{L^2(\partial F)}\) and \(D_F(F_\psi)\) together control the \(H^3\)-norm of \(\psi\).

\begin{lemma}
\label{Lh3control}
\ \\
Let \(F \subset \To^n\)  be a smooth critical set. There are positive constants \(K = K(F)\geq 1\) and
\(\delta = \delta(F)\) such that whenever \(\psi \in C^\infty(F)\) satisfies  \(\|\psi\|_{C^1(\partial F)} \leq \delta \), then
\[
K^{-1}\left( \|\nabla_{\tau}H_{F_\psi}\|_{L^2(\partial F_\psi)}+ \sqrt{D_F (F_\psi)} \right) \leq \|\psi\|_{H^3(\partial F)} \leq K\left( \|\nabla_{\tau}H_{F_\psi}\|_{L^2(\partial F_\psi)}+ \sqrt{D_F (F_\psi)} \right).
\] 
\end{lemma}

\begin{proof}
\ \\
We prove only the inequality
\[
\|\psi\|_{H^3(\partial F)} \leq K\left( \|\nabla_{\tau}H_{ F_\psi}\|_{L^2(\partial F_\psi)}+ \sqrt{D_F (F_\psi)} \right).
\]
The another one is easier to show. Again, it follows from \eqref{unifcontrol1b} and \eqref{l2equivalence}, that it suffices to find  positive \(\delta\) and \(K\) such that
for every \(\psi\in C^\infty(\partial F)\) with  \(\|\psi\|_{C^1(\partial E)} \leq \delta\)
\begin{equation}
\label{pLh3control1}
\|\psi\|_{H^3(\partial F)} \leq K\left( \|\nabla_{\tau}\left(H_{ F_\psi} \circ \Phi_\psi\right)\|_{L^2(\partial F)}+ \|\psi\|_{L^2(\partial F)} \right),
\end{equation}
where \(\Phi_\psi :\partial F \rightarrow \partial F_\psi\) is defined as in Lemma \ref{Lc1equivalence}. 
To this end we will prove the following auxiliary result. For suitable positive constants \(\delta'\) and \(K'\) 
the estimate
\begin{equation}
\label{pLh3control2}
\left|\nabla_\tau \left(H_{ F_\psi} \circ \Phi_\psi\right)\right|^2 \geq \frac{|\nabla_\co (\Delta_\co \psi)|^2}{2} -\frac{|\nabla_\co^3 \psi|^2}{4} 
 -K'\left( |\psi|^2+|\nabla_\co \psi|^2+|\nabla_\co^2 \psi|^2 + |\nabla_\co^2\psi|^4\right)
\end{equation}
holds on \(\partial F\) for every  \(\psi\in C^\infty(\partial F)\) with \(\|\psi\|_{C^1(\partial E)} \leq \delta'\).  Due to the compactness of \(\partial F\) we have only to show this holds locally.
Let \(\delta''\) be the constant for \(F\) as in Lemma \ref{LMC}.
Then for every \(\psi\in C^\infty(\partial F)\) with \(\|\psi\|_{C^1(\partial F)} \leq \delta''\)
\begin{align}
\notag
H_{F_\psi} \circ \Phi_\psi
&= - \Delta_\tau \psi + \langle A (\ \cdot \ ,\psi,\nabla_\tau \psi ),\nabla^2_\co \psi\rangle \\
\label{pLh3control3}
& +\nabla_\co \psi \left(Z(\ \cdot \ ,\psi,\nabla_\tau \psi)\right) +\psi P (\ \cdot \ ,\psi,\nabla_\tau \psi ) + H_F,
\end{align}
where
\begin{align*}
A&:\partial F \  \times [-\delta'',\delta''] \times [-\delta'',\delta'']^n  \rightarrow T^2_0(\partial F),\\
Z&:\partial F \ \times [-\delta'',\delta'']  \times [-\delta'',\delta'']^n  \rightarrow  T(\partial F) \ \ \ \text{and}\\
P&:\partial F \ \times [-\delta'',\delta'']  \times [-\delta'',\delta'']^n  \rightarrow \mathbb R
\end{align*}
are smooth maps depending on \(F\) and \(A(\ \cdot \ ,0,0):\partial F \rightarrow T^2_0(\partial F)\) is the zero tensor field. For a fixed point \(x_0 \in \partial F\)
choose a local orthonormal vector frame \((E_1,\ldots,E_{n-1})\) in some open neighborhood \(U_{x_0} \subset \partial F\) of \(x_0\).
Then we may write for every \(\psi\in C^\infty(\partial F)\) with \(\|\psi\|_{C^1(\partial E)} \leq \delta''\) 
\begin{align}
\label{pLh3control4}
\langle A (\ \cdot \ ,\psi,\nabla_\tau \psi ),\nabla^2_\co \psi\rangle &=  \sum_{ij} a_{ij} (\ \cdot \ ,\psi,\nabla_\tau \psi)\nabla^2_\co \psi(E_i,E_j) \ \ \ \text{and} \\
\label{pLh3control5}
\nabla_\co \psi \left(Z(\ \cdot \ ,\psi,\nabla_\tau \psi)\right) &=  \sum_{i} z_i (\ \cdot \ ,\psi,\nabla_\tau \psi)\nabla_\co \psi(E_i),
\end{align}
where \(a_{ij},z_i : U_{x_0} \ \times [-\delta'',\delta''] \times [-\delta'',\delta'']^n  \rightarrow \mathbb R\) are the smooth functions given
by 
\begin{align*}
a_{ij}(x,t,z) &=  A (x,t,z) (E_i(x),E_j(x)) \ \ \ \text{and} \\
z_i (x,t,z)  &= \langle E_i(x),Z(x,t,z)\rangle
\end{align*}
for every  \((x,t,z) \in  U_{x_0}  \times [-\delta'',\delta''] \times [-\delta'',\delta'']^n\).
Notice that then \(|(a_{ij}(x,t,z)| \leq |A(x,t,z)|\) and \(|z_i (x,t,z)| \leq |Z(x,t,z)|\).
By taking tangential gradient over \eqref{pLh3control3} (notice that for every \(\varphi \in C^\infty(\partial F)\) we may write \(\nabla_\tau\varphi = \sum_i (\nabla_{E_i} \varphi)  E_i\) in \(U_{p_0}\))
and using the expressions \eqref{pLh3control4} and \eqref{pLh3control5} we 
obtain
\begin{align}
\notag
\nabla_\tau (H_{F_\psi} \circ \Phi_\psi)
\notag
&= - \nabla_\tau (\Delta_\co \psi) +  \sum_{ijk} a_{ij} (\ \cdot \ ,\psi,\nabla_\tau \psi)\nabla^3_\co \psi(E_i,E_j,E_k) E_k \\
\notag
&+\sum_{ijk} a_{ij} (\ \cdot \ ,\psi,\nabla_\tau \psi)\left(\nabla^2_\co \psi(\nabla_{E_k}E_i,E_j) +  \nabla^2_\co \psi(E_i,\nabla_{E_k}E_j)\right)E_k \\
\notag
&+\sum_{ij} \nabla^2_\co \psi(E_i,E_j)\nabla_\tau\left(a_{ij} (\ \cdot \ ,\psi,\nabla_\tau \psi)\right) \\
\notag
&+\sum_{ik} z_i (\ \cdot \ ,\psi,\nabla_\tau \psi)\left(\nabla^2_\co \psi(E_i,E_k)+\nabla_\co \psi(\nabla_{E_k}E_i)\right)E_k  \\
\notag
&+\sum_{i} \nabla_\co \psi(E_i)\nabla_\tau \left(z_i (\ \cdot \ ,\psi,\nabla_\tau \psi)\right)  \\
\label{pLh3control6}
&+ P (\ \cdot \ ,\psi,\nabla_\tau \psi ) \nabla_\tau \psi + \psi\nabla_\tau\left(P (\ \cdot \ ,\psi,\nabla_\tau \psi ) \right).
\end{align}
Recall that \(\Delta_\co \psi = \Delta_\tau \psi\) and \(H_{F}\) is a constant thus vanishing after taking gradient over \eqref{pLh3control3}.
Again, for any smooth map \(u : U_{x_0}  \times [-\delta'',\delta''] \times [-\delta'',\delta'']^n\ \rightarrow \mathbb R\) one computes
\begin{align}
\notag
\nabla_\tau \left(u (\ \cdot \ ,\psi,\nabla_\tau \psi) \right) &= \nabla_\tau u(\ \cdot \ ,\psi,\nabla_\tau \psi) + \partial_t u (\ \cdot \ ,\psi,\nabla_\tau \psi) \nabla_\tau \psi \\
\label{pLh3control7}
& + \D^2_\tau \psi \nabla_n u(\ \cdot \ ,\psi,\nabla_\tau \psi) - \langle \nu_F, \nabla_n u(\ \cdot \ ,\psi,\nabla_\tau \psi) \rangle B_{F} \nabla_\tau \psi.
\end{align}
It follows from \(A(\ \cdot \ ,0,0)\) being the zero tensor field on \(\partial F\) and uniform continuity on compact sets, that there exist \(0<\delta'<\delta''\) and \(C\geq 1\) such that
for every \((x,t,z) \in \partial F  \times [-\delta',\delta'] \times [-\delta',\delta']^n\)
\begin{align}
\label{pLh3control8}
|A(x,t,z)| &\leq \frac{1}{64(n-1)^3} \ \ \ \text{and} \\
\label{pLh3control8b}
|Z(x,t,z)|,|P(x,t,z)| &\leq C.
\end{align}
By shrinking \(U_{x_0}\) and increasing \(C\), if necessary, we may assume that each \(|\nabla_{E_k} E_i|\) is bounded by \(C\) in \(U_{x_0}\).
Since \(|\nabla_\tau \psi| = |\nabla_\co\psi|\) and \(|\D_\tau^2 \psi| = |\nabla_\co^2\psi|\) on \(\partial F\) for every \(\psi \in C^\infty(\partial F)\), then again by shrinking \(U_{x_0}\) and increasing \(C\), if needed,
it follows from \eqref{pLh3control7} that for every \(\psi \in C^\infty(\partial F)\) with \(\|\psi\|_{C^1(\partial F)}\leq \delta'\)
\begin{equation}
\label{pLh3control9}
|\nabla_\tau \left(a_{ij} (\ \cdot \ ,\psi,\nabla_\tau \psi) \right)|,|\nabla_\tau \left(z_i (\ \cdot \ ,\psi,\nabla_\tau \psi) \right)|,|\nabla_\tau \left(P (\ \cdot \ ,\psi,\nabla_\tau \psi) \right)| \leq C(1+|\nabla_\co^2\psi|)
\end{equation}
in \(U_{x_0}\). Recalling expression \eqref{pLh3control6} for such \(\psi\) we denote 
\begin{align*}
a_\psi = \sum_{ijk} a_{ij} (\ \cdot \ ,\psi,\nabla_\tau \psi)\nabla^3_\co \psi(E_i,E_j,E_k) E_k
\end{align*}
and the sum of the lower order terms by \(b_\psi\). Thus we may write shortly in \(U_{x_0}\)
\[
\nabla_\tau (H_{F_\psi} \circ \Phi_\psi) = - \nabla_\tau(\Delta_\co \psi) + a_\psi + b_\psi.
\]
Now by using \eqref{pLh3control8} we have
\begin{equation}
\label{pLh3control10}
|a_\psi| \leq\frac{ |\nabla^3_\co \psi|}{64} 
\end{equation}
in \(U_p\). Again, by applying the estimates \eqref{pLh3control8b} and \eqref{pLh3control9} on \(b_\psi\) we find a positive constant \(C'\) depending on \(C\), \(\delta'\) and \(n\) such that in \(U_{x_0}\)
\begin{equation}
\label{pLh3control11}
|b_\psi| \leq C' \left(|\psi|+|\nabla_\co \psi|+|\nabla_\co^2 \psi| + |\nabla_\co^2\psi|^2\right).
\end{equation}
Thus by using \eqref{pLh3control10} and \eqref{pLh3control11} as well as the Cauchy-Schwarz and Young's inequalities, we have in \(U_{x_0}\)
\begin{align*}
|\nabla_\tau (H_{F_\psi} \circ \Phi_\psi)|^2 
&= |-\nabla_\tau(\Delta_\co \psi) + a_\psi + b_\psi|^2 \\
&\geq |\nabla_\tau(\Delta_\co \psi)|^2 - 2| \nabla_\tau(\Delta_\co \psi)| |  a_\psi + b_\psi| \\
&\geq \frac{1}{2}|\nabla_\tau(\Delta_\co \psi)|^2  -8| a_\psi + b_\psi|^2\\
&\geq \frac{1}{2}|\nabla_\tau(\Delta_\co \psi)|^2  -16| a_\psi|^2 -16| b_\psi|^2 \\
&\geq \frac{1}{2}|\nabla_\tau(\Delta_\co \psi)|^2- \frac{1}{4} |\nabla^3_\co \psi|^2 -64(C')^2\left(|\psi|^2+|\nabla_\co \psi|^2+|\nabla_\co^2 \psi|^2 + |\nabla_\co^2\psi|^4\right)\\
&= \frac{1}{2}|\nabla_\co(\Delta_\co \psi)|^2- \frac{1}{4} |\nabla^3_\co \psi|^2 -64(C')^2\left( |\psi|^2+|\nabla_\co \psi|^2+|\nabla_\co^2 \psi|^2+ |\nabla_\co^2\psi|^4\right).
\end{align*}
This implies \eqref{pLh3control2}. By integrating the both sides of \eqref{pLh3control2} over \(\partial F\) and applying \eqref{llaplace2} with the associated constant \(C_{\partial F}\)
we obtain
\begin{align}
\notag
\|\nabla_\tau (H_{F_\psi} \circ \Phi_\psi)\|_{L^2(\partial F)}^2 
&\geq \frac{1}{2}\|\nabla_\co(\Delta_\co \psi)\|_{L^2(\partial F)}^2- \frac{1}{4} \|\nabla^3_\co \psi\|_{L^2(\partial F)}^2 -K'\left( \|\psi\|_{H^2(\partial F)}^2 + \|\nabla_\co^2\psi\|_{L^4(\partial F)}^4\right) \\
\notag
&\geq \frac{1}{4}\|\nabla_\co^3\psi\|_{L^2(\partial F)}^2- (C_{\partial F}+ K')\left( \|\psi\|_{H^2(\partial F)}^2 + \|\nabla_\co^2\psi\|_{L^4(\partial F)}^4\right) \\
\notag
&\geq \frac{1}{16}\|\psi\|_{H^3(\partial F)}^2 +  \frac{3}{16}\|\nabla_\co^3\psi\|_{L^2(\partial F)}^2\\
\label{pLh3control12}
&- \left(C_{\partial F}+ K'+1\right)\left( \|\psi\|_{H^2(\partial F)}^2 + \|\nabla_\co^2\psi\|_{L^4(\partial F)}^4\right).
\end{align}
Next we interpolate the last terms in \eqref{pLh3control12}. By using Lemma \ref{Int1} and Lemma \ref{Int2} (recall that \(\nabla_\co(\nabla_\co \psi) = \nabla_\co^2 \psi\), \(\nabla_\co^2(\nabla_\co \psi)=\nabla^3_\co\psi\) and \(\|\nabla_\co\psi\|_{L^\infty(\partial F} \leq \|\psi\|_{C^1(\partial F)}\))  we find a positive \(M\) depending on \(F\) such that  
\begin{align}
\label{pLh3control13}
\|\nabla_\co \psi\|_{L^2(\partial F)}^2&\leq M \|\nabla_\co^3 \psi\|_{L^2(\partial F)}^\frac{2}{3}\|\psi\|_{L^2(\partial F)}^\frac{4}{3},\\
\label{pLh3control14}
\|\nabla_\co^2 \psi\|_{L^2(\partial F)}^2&\leq M \|\nabla_\co^3 \psi\|_{L^2(\partial F)}^\frac{4}{3}\|\psi\|_{L^2(\partial F)}^\frac{2}{3} \ \ \ \text{and} \\
\label{pLh3control15}
\|\nabla_\co^2 \psi\|_{L^4(\partial F)}^4 &\leq M \|\nabla_\co^3 \psi\|_{L^2(\partial F)}^2 \|\psi\|_{C^1(\partial F)}^2.
\end{align}
By applying Young's inequality on \eqref{pLh3control13} and \eqref{pLh3control14} we find \(K''= K''(K',M,C_{\partial F})\) such that
\begin{equation}
\label{pLh3control16}
\left(C_{\partial F}+ K'+1\right) \|\psi\|_{H^2(\partial F)}^2 \leq \frac{1}{8} \|\nabla_\co^3\psi\|_{L^2(\partial F)}^2 + K'' \|\psi\|_{L^2(\partial F)}^2.
\end{equation}
Finally by choosing \(0<\delta \leq \delta'\) with \(\left(C_{\partial F}+ K'+1\right)M \delta^2 \leq \frac{1}{16}\) it follows from \eqref{pLh3control12}, \eqref{pLh3control15} and \eqref{pLh3control16}, 
that for every \(\psi \in C^\infty(\partial F)\) with \(\|\psi\|_{C^1(\partial F)} \leq \delta\)
\[
\|\nabla_\tau (H_{F_\psi} \circ \Phi_\psi)\|_{L^2(\partial F)}^2 \geq \frac{1}{16}\|\psi\|_{H^3(\partial F)}^2 + K''\|\psi\|_{L^2(\partial F)}^2,
\]
which implies \eqref{pLh3control1}.
\end{proof}

The previous lemma and Lemma \ref{LMC} together yield, that for every smooth critical set \(F \subset \To^n\) there are \(\delta=\delta(F)\) and 
\(K=K(F)\) such that for every \(\psi \in C^\infty(\partial F)\) with \(\|\psi\|_{C^1(\partial F)} \leq \delta\)
\begin{align}
\label{h1MC}
\|\bar H_{ F_\psi} - H_{F_\psi}\|_{H^1(\partial F_\psi)}&\leq K \|\psi\|_{H^3(\partial F)} \\
\label{h1MCb}
\| H_{ F_\psi}\|_{H^1(\partial F_\psi)}&\leq K\left(\|\psi\|_{H^3(\partial F)} + |H_F| \right).
\end{align}

We will also use the following identities for the time derivatives of the \(L^2\)-norms of \(\bar H_t - H_t\) and \(\nabla_\tau H_t\).
For the proof see Appendix.
\begin{lemma}
\label{Ltimederivatives}
\ \\
Let \((E_t)_t\) be a volume preserving mean curvature flow in \(\To^n\). Then for every \(0\leq t < T^*\)
\begin{align}
\label{ltimederivatives1}
\frac{\d}{\d t}  \|\bar H_t - H_t\|_{L^2(\partial E_t)}^2 
 = &-2 \partial^2 P(E_t)[\bar H_t - H_t]+ \int_{\partial E_t} H_t(\bar H_t - H_t)^3 \ \d \mathcal H^{n-1}  \ \ \text{and}\\
\notag
\frac{\d}{\d t}\|\nabla_{\tau}H_t\|_{L^2(\partial E_t)}^2
=&-2\int_{\partial E_t} (\Delta_\tau H_t)^2 -(\bar H_t - H_t)|B_t|^2 \Delta_\tau H_t \ \d \mathcal H^{n-1} \\
\notag
&-2\int_{\partial E_t} (\bar H_t - H_t) \langle\nabla_\tau H_t, B_t \nabla_\tau H_t\rangle \ \d \mathcal H^{n-1} \\
\label{ltimederivatives2}
&+\int_{\partial E_t} |\nabla_\tau H_t|^2 (\bar H_t - H_t) H_t \  \d \mathcal H^{n-1}.
\end{align}
\end{lemma}

We are now ready to prove Theorem \ref{M1}. We divide it into four steps. 
\begin{proof}[Proof of Theorem \ref{M1}]
\ \\
\textbf{Step 1.}
We want first to utilize several lemmas and estimates we have gathered by controlling \(C^1\)- and \(W^{2,p}\)-norms for \(1\leq p \leq 6\).
Indeed, by Remark \ref{H3Remark}  there is a constant \(K_0 \in \mathbb R_+\) depending on \(F\), such that for every \(\psi \in C^\infty(\partial F)\) and \(1\leq p \leq 6\)
\begin{equation}
\label{M1p1}
\|\psi\|_{C^1(\partial F)}, \|\psi\|_{W^{2,p}(\partial F)} \leq K_0 \|\psi\|_{H^3(\partial F)}.
\end{equation}
Notice that the assumption \(n\leq 4\) is really needed for this conclusion. 
Using the estimate \eqref{M1p1} we find  \(0< \delta < 1\) and \(1<K<\infty\)  so 
that for any \( \psi \in C^\infty(\partial F)\) with \( \|\psi\|_{H^3(\partial F)} \leq \delta\) the norm \(\|\psi\|_{C^1(\partial F)}\) is so small that
the following four conditions hold.
\begin{itemize}
\item[(i)]
 Lemma \ref{Lc1equivalence} is satisfied, i.e., \(F_\psi\) is a well-defined smooth set. Moreover, we may assume
\begin{equation}
\label{M1p2}
\|F_\psi,F\|_{C^1} \leq \frac{\delta_F}{2},
\end{equation}
where \(\delta_F\) is as in the lemma.
\item[(ii)]
The inequalities \eqref{unifcontrol1a} (for \(\bar d_F\), then \(\bar d_F \circ \Phi_\psi = \psi\)) and \eqref{l2equivalence}  are satisfied with \(K\), i.e.,
\begin{align}
\label{M1p3}
K^{-1}\|\psi\|_{L^2(\partial F)} \leq \|\bar d_F\|_{L^2(\partial F_\psi)}, \sqrt{D_F(F_\psi)} \leq K \|\psi\|_{L^2(\partial F)}.
\end{align}
\item[(iii)]
Lemma \ref{Lh3control} and \eqref{h1MC} are satisfied with \(K\), i.e.,
\begin{align}
\label{M1p4}
\|\psi\|_{H^3(\partial F)} &\leq K \left(\|\nabla_\tau H_{F_\psi} \|_{L^2(\partial F_\psi)}+ \sqrt{D_F(F_\psi)}\right) \\
\label{M1p5}
\|\bar H_{F_\psi} - H_{F_\psi} \|_{H^1(\partial F_\psi)} &\leq K \|\psi\|_{H^3(\partial F)}.
\end{align}

\item[(iv)]
Lemma \ref{LC1unifcontrol} is satisfied for \(\bar H_{F_\psi} - H_{F_\psi}\) and \(1\leq p \leq 6\) with \(K\), i.e.,
\begin{equation}
\label{M1p6}
\|\bar H_{F_\psi} - H_{F_\psi} \|_{L^p(\partial F_\psi)} \leq K \|\bar H_{F_\psi} - H_{F_\psi} \|_{H^1(\partial F_\psi)}.
\end{equation}
\end{itemize}

Moreover, we may assume \(\delta\) to be so small and \(K\) to be so large that by \eqref{M1p1}, \eqref{lMC2}, Lemma \ref{LH3unifcontrol} and Lemma \ref{Lsvar} (note that \( \bar H_{F_\psi} - H_{F_\psi} \in  T^{\perp,L^2}(\partial F_\psi)\))

\begin{align}
\label{M1p7}
\|\nabla_\tau H_{F_\psi}\|_{L^4(\partial F_\psi)} & \leq K \left(\| \Delta_\tau H_{F_\psi} \|_{L^2(\partial F_\psi)} +  \|\bar H_{F_\psi} - H_{F_\psi} \|_{H^1(\partial F_\psi)}\right), \\
\label{M1p8}
\|H_{F_\psi}\|_{L^p(\partial F_\psi)},\|B_{F_\psi}\|_{L^p(\partial F_\psi)} &\leq K \ \ \ \text{for} \ \ \ 1 \leq p \leq 6 \ \ \ \text{and}  \\
\label{M1p9}
\sigma_1 \|\bar H_{F_\psi} - H_{F_\psi}\|^2_{H^1(\partial F_\psi)} &\leq \partial^2 P(F_\psi)[\bar H_{F_\psi} - H_{F_\psi} ].
\end{align} 

Next we fix the constants \(C_0\) and \(\epsilon_0\) by setting
\begin{align}
\label{M1p10}
C_0 &= \frac{4K^6 + 1}{\sigma_1} \\
\label{M1p11}
\epsilon_0 &= \min\left\{\delta,\frac{\min\{\sigma_1,1\}}{16 K^6}\right\}.
\end{align}
Again, for given \(0<\epsilon \leq \epsilon_0\) set \(\gamma_\epsilon\) to be the maximal \(0< s \leq  \frac{\epsilon}{2}\) satisfying
\begin{equation}
\label{M1p12}
\frac{2}{\sigma_1} \sqrt s  + K^2 s^2 \leq \frac{\epsilon^2}{16K^2} \ \ \ \text{and} \ \ \  K^2\left(1 + C_0\right) s \leq \frac{\epsilon}{4}.
\end{equation}
\textbf{Step 2.}
Suppose that \((E_t)_t\) is the volume preserving mean curvature flow with a smooth initial datum \(E_0=F_{\psi_0}\), where \(\psi_0 \in C^\infty(\partial F)\) and \(\|\psi_0\|_{H^3(\partial F)} \leq \gamma_\epsilon\).
Since \(t \mapsto \|E_t,F\|_{C^1}\) is continuous on \([0,T^*[\) (recall \(T^*\) is the maximal lifetime), then it follows from Lemma 
\ref{Lc1equivalence} and \eqref{M1p2} that we may write \(E_t = F_{\psi_t}\) for unique \(\psi_t \in C^\infty(\partial F)\) with 
%continuous\(t \mapsto \|\psi_t\|_{H^3(\partial F)}\) 
$t \mapsto \psi_t$ smoothly parametrized over a short time period. Hence
\[
T_\epsilon = \sup\left\{s \in [0,T^*[ \ : E_t = F_{\psi_t} \  \text{for} \ \psi_t \in C^\infty(\partial F), \ \|\psi_t\|_{H^3(\partial F)} \leq \epsilon \ \forall t \in  [0,s] \right\}
\]
must be a positive number. The key idea is to show that the claim of theorem is satisfied for \(\epsilon\) on the time interval \([0,T_\epsilon[\) and by virtue of the choice of \(\gamma_\epsilon\) we have in fact 
\begin{equation}
\label{M1p13}
\|\psi_t\|_{H^3(\partial F)} \leq \frac{\epsilon}{2} \ \ \ \text{on} \ \ \ [0,T_\epsilon[.
\end{equation}
By using a similar continuity argument as before one shows that the condition \eqref{M1p13} implies \(T_\epsilon = T^*\).
\begin{comment}
Suppose that the previous holds and by contradiction suppose that \(T_\epsilon < T^*\). Since now \(t \mapsto \|E_t,F\|_{C^1(\To^n)}\) is continuous on \([0,T^*[\) 
and \(\|E_t,F\|_{C^1(\To^n)} \leq \frac{\delta_F}{2}\) on \([0,T_\epsilon[\), there is \(0<s < T^*-T_\epsilon\) so that \(\|E_t,F\|_{C^1} \leq \delta_F\) for every \(0\leq t<T_\epsilon + s\). Hence it follows from Lemma \ref{Lc1equivalence}, that we can write \(E_t = F_{\psi_t}\) with a unique
 \(\psi_t \in C^\infty(\partial F)\) for every \(0\leq t<T_\epsilon + s\) and the map \(t \mapsto \|\psi_t\|_{H^3(\partial F)}\) is continuous on 
\([0,T_\epsilon + s[\). But that would contradict the maximality of \(T_\epsilon\) and hence must hold \(T_\epsilon = T^*\). 
\end{comment}
\newline
\newline
\textbf{Step 3.}
For every \(t \in \ [0,T_\epsilon[\)
\begin{align}
\notag
&\frac{\d}{\d t}  \|\bar H_t - H_t\|_{L^2(\partial E_t)}^2 \\
\notag
&\overset{\eqref{ltimederivatives1}}{=} -2 \partial^2 P(E_t)[\bar H_t - H_t]+ \int_{\partial E_t} H_t(\bar H_t - H_t)^3 \ \d \mathcal H^{n-1} \\
\notag
&\overset{\eqref{M1p9}}{\leq} - 2\sigma_1 \|\bar H_t - H_t\|_{H^1(\partial E_t)}^2 + \int_{\partial E_t} H_t(\bar H_t - H_t)^3 \ \d \mathcal H^{n-1}\\
\notag
&\leq - 2\sigma_1 \|\bar H_t - H_t\|_{H^1(\partial E_t)}^2 + \|H_t\|_{L^4(\partial E_t)}\|\bar H_t - H_t\|_{L^4(\partial E_t)}^3\\
\notag
&\overset{\eqref{M1p8}}{\leq} - 2\sigma_1 \|\bar H_t - H_t\|_{H^1(\partial E_t)}^2 + K \|\bar H_t - H_t\|_{L^4(\partial E_t)}^3\\
\notag
&\overset{\eqref{M1p6}}{\leq}- 2\sigma_1 \|\bar H_t - H_t\|_{H^1(\partial E_t)}^2 + K^4 \|\bar H_t- H_t\|_{H^1(\partial E_t)}^3 \\
\notag
&\overset{\eqref{M1p5}}{\leq} \left(- 2\sigma_1 + K^5 \|\psi_t\|_{H^3(\partial F)}\right) \|\bar H_t - H_t\|_{H^1(\partial E_t)}^2  \\
\notag
&\leq \left(- 2\sigma_1 + K^5 \epsilon \right) \|\bar H_t - H_t\|_{H^1(\partial E_t)}^2  \\
\label{M1p14}
&\overset{\eqref{M1p11}}{\leq}- \sigma_1 \|\bar H_t - H_t\|_{H^1(\partial E_t)}^2.
\end{align}
Since \(- \sigma_1 \|\bar H_t - H_t\|_{H^1(\partial E_t)}^2 \leq - \sigma_1 \|\bar H_t - H_t\|_{L^2(\partial E_t)}^2\), by using
Grönwall's lemma we obtain
\begin{align}
\notag
\|\bar H_t - H_t\|^2_{L^2(\partial E_t)} 
&\leq \|\bar H_0 - H_0\|^2_{L^2(\partial E_t)} e^{-\sigma_1 t} \\
\notag
&\overset{\eqref{M1p5}}{\leq} K^2 \|\psi_0\|^2_{H^3(\partial F)} e^{-\sigma_1 t} \\
\label{M1p15}
&\overset{\eqref{M1p11}}{\leq} \|\psi_0\|_{H^3(\partial F)} e^{- \sigma_1 t}
\end{align} 
so \eqref{m1} holds on \([0,T_\epsilon[\). Next we estimate \( D_F(E_t)\) on \([0,T_\epsilon[\). For the time derivative
\begin{align*}
\frac{\d}{\d t} D_F(E_t) &\overset{\eqref{Dtime}}{=}\int_{\partial E_t} \bar d_F (\bar H_t - H_t) \d \mathcal H^{n-1} \\
&\leq  \|\bar d_F\|_{L^2(\partial E_t)} \|\bar H_t - H_t\|_{L^2(\partial E_t)} \\
&\overset{\eqref{M1p3}}{\leq}  K \|\psi_t\|_{L^2(\partial F)} \|\bar H_t - H_t\|_{L^2(\partial E_t)}\\
&\leq  K\epsilon \|\bar H_t - H_t\|_{L^2(\partial E_t)}\\
&\overset{\eqref{M1p15}}{\leq} K\epsilon \|\psi_0\|_{H^3(\partial F)}^\frac{1}{2} e^{-\frac{1}{2} \sigma_1 t} \\
&\overset{\eqref{M1p11}}{\leq}  \|\psi_0\|_{H^3(\partial F)}^\frac{1}{2}  e^{-\frac{1}{2} \sigma_1 t}.
\end{align*}
Thus integrating over time yields 
\begin{align}
\notag
 D_F(E_t) 
&\leq \frac{2}{\sigma_1} \|\psi_0\|_{H^3(\partial F)}^\frac{1}{2}  \left(1- e^{-\frac{\sigma_1}{2}}\right) +  D_F(E_0) \\
\notag
&\overset{\eqref{M1p3}}{\leq} \frac{2}{\sigma_1}\|\psi_0\|_{H^3(\partial F)}^\frac{1}{2} + K^2 \|\psi_0\|_{L^2(\partial F)}^2  \\
\label{M1p16}
&\overset{\eqref{M1p12}}{\leq} \frac{\epsilon^2}{16K^2}.
\end{align}
\textbf{Step 4.} In this last step we finish the proof by showing that  \eqref{m2}  and \eqref{M1p13} are satisfied on \([0,T_\epsilon[\). 
To this end we have to estimate  \(\frac{\d}{\d t}\|\nabla_{\tau}H_t\|_{L^2(\partial E_t)}^2\) on \([0,T_\epsilon[\).
Recall that by \eqref{ltimederivatives2} we have
\begin{align*}
\frac{\d}{\d t}\|\nabla_{\tau}H_t\|_{L^2(\partial E_t)}^2
&=-2\int_{\partial E_t} (\Delta_\tau H_t)^2 -(\bar H_t - H_t)|B_t|^2 \Delta_\tau H_t \ \d \mathcal H^{n-1} \\
&-2\int_{\partial E_t} (\bar H_t - H_t) \langle\nabla_\tau H_t, B_t \nabla_\tau H_t\rangle \ \d \mathcal H^{n-1} 
+\int_{\partial E_t} |\nabla_\tau H_t|^2 (\bar H_t - H_t) H_t \  \d \mathcal H^{n-1} \\
&=T_1+T_2+T_3.
\end{align*}
Next we estimate the terms \(T_1\), \(T_2\) and \(T_3\).
First we have 
\begin{align}
\notag
T_1&\overset{\text{Young}}{\leq}-\|\Delta_\tau H_t\|_{L^2(\partial E_t)}^2  +4 \int_{\partial E_t} |\bar H_t - H_t|^2|B_t|^4 \ \d \H^{n-1}\\
\notag
&\leq -\|\Delta_\tau H_t\|_{L^2(\partial E_t)}^2  +4 \|\bar H_t - H_t\|_{L^6(\partial E_t)}^2 \|B_t\|_{L^6(\partial E_t)}^4\\
\notag
&\overset{\eqref{M1p8}}{\leq}-\|\Delta_\tau H_t\|_{L^2(\partial E_t)}^2  +4 K^4 \|\bar H_t - H_t\|_{L^6(\partial E_t)}^2 \\
\label{M1p17}
& \overset{\eqref{M1p6}}{\leq} -\|\Delta_\tau H_t\|_{L^2(\partial E_t)}^2  +4 K^6\|\bar H_t - H_t\|_{H^1(\partial E_t)}^2.
\end{align}
Second
\begin{align}
\notag
\notag
T_2&\leq 2 \int_{\partial E_t} |\bar H_t - H_t| |\nabla_\tau H_t|^2 | B_t | \ \d \mathcal H^{n-1}\\
\notag
&\leq 2\|\nabla_\tau H_t\|_{L^4(\partial E_t)}^2  \| \bar H_t- H_t\|_{L^4(\partial E_t)}\|B_t\|_{L^4(\partial E_t)} \\
\notag
&\overset{\eqref{M1p8}}{\leq} 2 K \|\nabla_\tau H_t\|_{L^4(\partial E_t)}^2  \| \bar H_t- H_t\|_{L^4(\partial E_t)}\\
\notag
&\overset{\eqref{M1p7}}{\leq} 2 K^3 \left(\| \Delta_\tau H_t \|_{L^2(\partial E_t)} +  \|\bar H_t - H_t \|_{H^1(\partial E_t)}\right)^2 \| \bar H_t- H_t\|_{L^4(\partial E_t)}\\
\notag
&\overset{\eqref{M1p6}}{\leq} 2 K^4 \| \bar H_t- H_t\|_{H^1(\partial E_t)} \left(\| \Delta_\tau H_t \|_{L^2(\partial E_t)} +  \|\bar H_t - H_t \|_{H^1(\partial E_t)}\right)^2 \\
\notag
&\overset{\eqref{M1p5}}{\leq} 2 K^5 \|\psi_t\|_{H^3(\partial F)} \left(\| \Delta_\tau H_t \|_{L^2(\partial E_t)} + \|\bar H_t - H_t \|_{H^1(\partial E_t)}\right)^2 \\
\notag
&\leq 4 K^5 \|\psi_t\|_{H^3(\partial F)}\| \Delta_\tau H_t \|_{L^2(\partial E_t)}^2  +  4 K^5 \|\psi_t\|_{H^3(\partial F)} \|\bar H_t - H_t \|_{H^1(\partial E_t)}^2 \\
\notag
&\leq 4 K^5 \epsilon \| \Delta_\tau H_t \|_{L^2(\partial E_t)}^2  +  4 K^5 \epsilon \|\bar H_t - H_t \|_{H^1(\partial E_t)}^2 \\
\label{M1p18}
&\overset{\eqref{M1p11}}{\leq} \frac{1}{4}\|\Delta_\tau H_t \|_{L^2(\partial E_t)}^2  +  \frac{1}{2} \|\bar H_t - H_t \|_{H^1(\partial E_t)}^2 .
\end{align}
Finally by estimating in a similar way as above we obtain
\begin{equation}
\label{M1p19}
T_3\leq \frac{1}{8}\|\Delta_\tau H_t \|_{L^2(\partial E_t)}^2  +  \frac{1}{4} \|\bar H_t - H_t \|_{H^1(\partial E_t)}^2.
\end{equation}
Hence \eqref{M1p17}, \eqref{M1p18} and \eqref{M1p19} together yield
\begin{align*}
\frac{\d}{\d t}\|\nabla_{\tau}H_t\|_{L^2(\partial E_t)}^2
&\leq -\frac{1}{2}\| \Delta_\tau H_t \|_{L^2(\partial E_t)}^2  + \left(4K^6 + 1\right) \|\bar H_t - H_t\|_{H^1(\partial E_t)}^2 \\
&\overset{\eqref{M1p10}}{=} -\frac{1}{2}\| \Delta_\tau H_t \|_{L^2(\partial E_t)}^2  +  \sigma_1 C_0 \|\bar H_t - H_t\|_{H^1(\partial E_t)}^2 
\end{align*}
on \([0,T_\epsilon[\). Then the previous estimate with \eqref{M1p14} yields that for every \(t \in [0,T_\epsilon[\)
\begin{equation*}
\label{M1p20}
\frac{\d}{\d t}\left[ \|\nabla_{\tau}H_t\|_{L^2(\partial E_t)}^2+ C_0 \|\bar H_t - H_t\|_{L^2(\partial E_t)}^2 \right]
\leq -\frac{1}{2}\| \Delta_\tau H_t \|_{L^2(\partial E_t)}^2,
\end{equation*}
which implies \eqref{m2}. In particular
\[
 t \mapsto \|\nabla_{\tau}H_t\|_{L^2(\partial E_t)}^2+ C_0 \|\bar H_t - H_t\|_{L^2(\partial E_t)}^2 
\]
\ \\
is decreasing map on \([0,T_\epsilon[\) and therefore
\begin{equation}
\label{M1p20}
\|\nabla_\tau H_t\|_{L^2(\partial E_t)} \leq \|\nabla_\tau H_0\|_{L^2(\partial E_0)} + C_0 \|\bar H_0-H_0\|_{L^2(\partial E_0)}.
\end{equation}
Finally for every \(t \in  [0,T_\epsilon[\) 
\begin{align*}
\|\psi_t\|_{H^3(\partial E_t)}  
&\overset{\eqref{M1p4}}{\leq}
 K\left(\|\nabla_\tau H_t\|_{L^2(\partial E_t)} + \sqrt{D_F(E_t)}\right) \\
&\overset{\eqref{M1p20}}{\leq}
K\left(\|\nabla_\tau H_0\|_{L^2(\partial E_0)} + C_0 \|\bar H_0-H_0\|_{L^2(\partial E_0)}\right) + K \sqrt{D_F(E_t)} \\
&\overset{\eqref{M1p16}}{\leq}
 K\left(1 + C_0\right) \|\bar H_0-H_0\|_{H^1(\partial E_0)} + \frac{\epsilon}{4} \\
&\overset{\eqref{M1p5}}{\leq}
 K^2\left(1 + C_0\right) \|\psi_0\|_{H^3(\partial E_t)} +\frac{\epsilon}{4} \\
&\overset{\eqref{M1p12}}{\leq}
 \frac{\epsilon}{2}.
\end{align*}
\end{proof}

%%%%%%%%%%%%%%%%%%%%%%%%%%%%%%%%%%%%%%%%%%%%%%%%%%%%%%%%%%%%%%%%%%%%%%%%%%%%%%%%%%%%%%%%%%%%%%%%%%%%%%%%%%%%%%%%%%%%%%

\section{The main result}

In this section we will prove the main result. We give first the technical statement of the theorem in contrast to the heuristical one we presented in Introduction.

\begin{thm}[\textbf{Main Theorem}]
\label{MT}
\ \\
Let \(\mathbb T^n\) be a flat torus with \(n=3,4\) and assume that \(F \subset \mathbb T^n\) is a strictly stable set. There exist positive
constants \(\delta_0, \sigma_0 \in \mathbb R_+\) depending on \(F\) such that the following hold.

If \(E_0\) is a smooth set in \(\mathbb T^n\) with \(|E_0|=|F|\) of the form \(E_0 = F_{\psi_0}\),
where \(\psi_0 \in C^\infty(\partial F)\) and \(\|\psi_0\|_{H^3(\partial F)} \leq \delta_0 \), then the volume preserving mean curvature flow 
\((E_t)_t\) in \(\mathbb T^n\) with the initial datum \(E_0\) satisfies the following conditions.
\begin{itemize}
\item[(i)]  The flow has infinite lifetime.
\item[(ii)] There exist \(p = p(F,E_0) \in \mathbb R^n\) and \(C = C(F) \in \mathbb R_+\)  such that the flow converges to \(F+p\) exponentially fast in \(W^{2,5}\)-sense, that is,
\(E_t = (F+p)_{\varphi_t}\) for \(\varphi_t \in C^\infty(\partial(F+p))\) and \(\|\varphi_t\|_{W^{2,5}} \leq C e^{-\sigma_0 t}\).
\item[(iii)] \(|p| \rightarrow 0\) and \(C \rightarrow 0\) as \( \|\psi_0\|_{H^3(\partial F)} \rightarrow 0\).
\end{itemize}
\end{thm}

\begin{rem}
\label{MT2}
\ \\
In the statement of the main theorem the \(W^{2,5}\)-convergence can be replaced by \(W^{2,q}\)-convergence, where
\(1\leq q<\infty\), if \(n=3\), and \(1\leq q<6\), if \(n=4\). In this case the proof would be similar to the original proof.
\end{rem}

The main idea of the proof is obviously to employ the short time existence (Theorem \ref{Ex}) and the monotonicity result (Theorem \ref{M1}). 

\begin{proof}[Proof of the Main Theorem]
\ \\
 Let  \(F \subset \To^n\) be a strictly stable set and let us fix the constants \(\delta_0\) and 
\(\sigma_0\) for \(F\) as in the statement of the main theorem. Let \(\epsilon_0,\sigma_1 \in \mathbb R_+\) and 
\(0<\gamma_\epsilon < \epsilon\) (for every \(0<\epsilon\leq \epsilon_0\)) for \(F\)  be as in Theorem \ref{M1}. 
Choose first a positive \(c\) so small that the following hold. 
\begin{itemize}
\item[(i)]
The condition \(\|\psi\|_{C^1(\partial F)} \leq c\) for \(\psi \in C^{1,\frac{1}{4}}(\partial F)\)  implies via Lemma \ref{Lc1equivalence} that the set \(E_\psi\) is defined as a \(C^{1,\frac{1}{4}}\)-set, the map \(\Phi_\psi\) defined as in the same lemma is a \(C^{1,\frac{1}{4}}\)-diffeomorphism
from \(\partial F\) to \(\partial F_\psi\) and
\begin{equation}
\label{pMT1}
 K^{-1} |\partial F| \leq |\partial F_\psi| \leq K |\partial F|
\end{equation}
for some real number \(K \geq 1\) depending on \(c\) and \(F\).
\item[(ii)]
 Further, if \(\psi\) is smooth, the same condition implies via \eqref{h1MCb} that by increasing 
\(K\), if necessary, 
\begin{equation}
\label{pMT2}
\| H_{F_\psi}\|_{H^1( F_\psi)} \leq K\left(\|\psi\|_{H^3(\partial F)}+|H_F| \right).
\end{equation}
\item[(iii)]
Since a translate of \(F\) satisfies Lemma \ref{Lh3control} and \eqref{l2equivalence} with the same bounds as \(F\), then by using the previous lemmas and
Lemma \ref{Lc1equivalence}, possibly decreasing \(c\) and increasing \(K\), we obtain the following.
Assume \(F+p = F_g\) with \(g \in C^\infty(\partial F)\) and \(\|g\|_{C^1(\partial F)} \leq c\). Then for every \(\psi \in C^\infty(\partial F)\) 
with \(\|\psi\|_{C^1(\partial F)} \leq c\)
there is a unique \(\varphi \in C^\infty(\partial (F + p))\) such that \(F_\psi = (F + p)_\varphi\) and we have the following uniform estimates
\begin{align}
\label{pMT3}
K^{-1} \sqrt{D_{F + p}(F_\psi)} \leq \|\varphi\|_{L^2(\partial (F + p))} &\leq K \sqrt{D_{F + p}(F_\psi)} \ \ \ \text{and} \\
\label{pMT4}
K^{-1} \|\nabla_\tau H_{ F_\psi}\|_{L^2(\partial F_\psi)} \leq \|\varphi\|_{H^3(\partial(F + p))} &\leq K \left(\|\nabla_\tau H_{F_\psi}\|_{L^2(\partial F_\psi)}+ \sqrt{D_{F + p}(F_\psi)}\right).
\end{align}

\item[(iv)]
Finally the condition \(\|\psi\|_{C^{1,\frac{1}{4}}(\partial F)} \leq c\) for \(\psi \in C^\infty(\partial F)\) means that Theorem \ref{Ex} (in the case
\(\beta=\frac{1}{4}\)) is satisfied for the initial set \(E_\psi\). 
\end{itemize}
Next choose \(0<\epsilon_1 \leq \epsilon_0\) such that the condition \(\|\psi\|_{H^3(\partial F)} \leq \epsilon_1\) for \(\psi \in C^\infty(\partial F)\) means that
 first Lemma \ref{Ltransl0} is satisfied, provided that \(F_\psi\) is critical with \(|F_\psi|=|F|\), and second via Lemma 
\ref{Emb} \(\|\psi\|_{C^1(\partial F)} \leq \|\psi\|_{C^{1,\frac{1}{4}}(\partial F)} \leq c\) (here we really need the assumption \(n \leq 4\)).
At this point we set 
\begin{equation}
\label{pMT5}
\delta_0 = \gamma_{\epsilon_1} \ \ \ \text{and} \ \ \ \sigma_0 = -\frac{\sigma_1}{2}\left(\frac{1}{3} - \frac{n-1}{10}\right).
\end{equation}
Fix an arbitrary \(\psi_0 \in C^\infty(\partial F)\) with \(\|\psi_0\|_{H^3(\partial F)} \leq \delta_0\) and \(|F_{\psi_0}|=|F|\). Then by Theorem \ref{Ex} there exists a unique volume preserving mean curvature flow
\((E_t)_t\) starting from \(E_0 = F_{\psi_0}\) and the maximal lifetime \(T^*\) is bounded from below by \(T = T(F,\frac{1}{4})>0\) as in Theorem \ref{Ex}. Again, by Theorem \ref{M1}  we may write \(E_t = F_{\psi_t}\), where
\(\psi_t \in C^\infty(\partial F)\) with \(\|\psi_t\|_{H^3(\partial F)} \leq \epsilon_1\), and the inequalities \eqref{m1}
and \eqref{m2} are satisfied for every \(t \in [0,T^*[\). We divide the proof into three steps, whose statements are listed below.

\begin{itemize}

\item[\textbf{Step 1.}]
The flow \((E_t)_t\) has infinite lifetime and there exists \(\psi_\infty \in H^3(\partial F)\) with \(\|\psi_\infty\|_{H^3(\partial F)} \leq \epsilon\)
such that \(\psi_t \rightarrow \psi_\infty\) in \(C^{1,\frac{1}{4}}(\partial F)\). Further, there exist a positive constant \(C_F\) independent of the choice of \(\psi_0\) 
%satisfying \(\|\psi_0\|_{H^3(\partial F)} \leq \delta_0\) 
and an increasing \(\rho:[0,\delta_0] \rightarrow [0,\infty[\) with \(\lim_{s \rightarrow 0 +} \rho(s) = 0\) such that
 every \(t \in [0,\infty[\)
\begin{align}
\label{pMT6}
D_{E_\infty} (E_t) &\leq C_F \|\psi_0\|_{H^3(\partial F)}  e^{-\sigma_1 t} \ \ \ \text{and} \\
\label{pMT7}
\|\psi_\infty\|_{L^\infty(\partial F)} &\leq \rho\left(\|\psi_0\|_{H^3(\partial F)}\right),
\end{align}
where \(E_\infty = F_{\psi_\infty}\) is the corresponding \(C^{1,\frac{1}{4}}\)-limit set.

\item[\textbf{Step 2.}]
The limit set is of the form \(E_\infty = F + p\), where \(p \rightarrow 0\) as \(\|\psi_0\|_{H^3(\partial F)} \rightarrow 0\).

\item[\textbf{Step 3.}]
The \(W^{2,5}\)-convergence of the flow: For each \(t \in [0,\infty[\) there is \(\varphi_t \in C^\infty(\partial E_\infty)\) with \(E_t = (E_\infty)_{\varphi_t}\) and 
\(\|\varphi_t\|_{W^{2,5}(\partial F)} \leq C e^{-\sigma_0 t}\), where \(C\) is independent of the choice of \(\psi_0\) and \(C \rightarrow 0\) as \(\|\psi_0\|_{H^3(\partial F)} \rightarrow 0\).

\end{itemize}

Since \(\psi_0 \in C^\infty(\partial F)\) with \(\|\psi_0\|_{H^3(\partial F)} \leq \delta_0\) was arbitrarily chosen,  the claim of theorem follows immediately from these statements. 
Let us prove them in order as listed.
\newline
\newline
\emph{Proof of Step 1.} Assume by contradiction \(T^*<\infty\) and choose \(\hat t \in \ [0,T^*[\) such that \(T^*-\hat t < T\), where 
\(T=T(F,\frac{1}{4})\) as in Theorem \ref{Ex}. Now \(\|\psi_{\hat t}\|_{H^3(\partial F)} \leq \epsilon_1 \) so \(\|\psi_{\hat t}\|_{C^{1,\frac{1}{4}}(\partial F)} \leq c\) and hence by Theorem \ref{Ex} there exists a unique volume preserving mean curvature flow \((\hat E_t)_t\) starting from \(E_{\hat t}\) with a maximal lifetime at least \(T\). It follows from the uniqueness and from the semi-group property of \((E_t)_t\) that \(E_t=\hat E_{t-\hat t}\) for every \(t \in [\hat t,T^*[\) . This means that the flow \((E_t)_t\) can be extended beyond \(T^*\), which contradicts  its maximality. Therefore it holds \(T^*=\infty\).

Take a sequence \((t_k)_{k=1}^\infty \subset \mathbb R_+\) with \(t_k \rightarrow \infty\). Since \(\|\psi_{t_k}\|_{H^3(\partial F)} \leq \epsilon_1\) for every \(k\) and \( H^3(\partial F)\) is weakly compact, 
then, up to a subsequence, there is a weak limit \(\psi_\infty \in H^3(\partial F)\) with \(\|\psi_{\infty}\|_{H^3(\partial F)} \leq \epsilon_1\). 
Further, it follows from Lemma \ref{Emb} that the sequence converges to \(\psi_\infty\) in \(C^{1,\frac{1}{4}}(\partial F)\). Now \(C^1\)-convergence implies that \(\|\psi_{\infty}\|_{C^1(\partial F)} \leq c\) so \(E_\infty := F_{\psi_\infty}\) is defined as a \(C^{1,\frac{1}{4}}\)-set and the map \(\Phi_{\psi_\infty}\) is a \(C^{1,\frac{1}{4}}\)-diffeomorphism
from \(\partial F\) to \(\partial E_\infty\).

Next we check that \eqref{pMT6} holds for every \(t\). Notice first that  \(|E_{t_k} \Delta E_\infty| \rightarrow 0\), which implies
\(D_{E_\infty} (E_{t_k}) \rightarrow 0\). For a fixed \(t \in [0,\infty[\) and every \(t_k > t\) we may estimate
\begin{align*}
|D_{E_\infty} (E_{t_k})-D_{E_\infty} (E_t)| 
&= \left| \int_t^{t_k} \frac{\d}{\d s} D_{E_\infty} (E_s) \ \d s\right| \\
&\overset{\eqref{Dtime}}{=} \left| \int_t^{t_k}\int_{\partial E_s} \bar d_{E_\infty} (\bar H_s -H_s) \ \d \H^{n-1} \d s\right| \\
&\leq \int_t^{t_k}\|\bar d_{E_\infty}\|_{L^2(\partial E_s)} \|\bar H_s -H_s\|_{L^2(\partial E_s)} \ \d s \\
&\leq \sqrt n \int_t^{t_k}|\partial E_s| \|\bar H_s -H_s\|_{L^2(\partial E_s)} \ \d s \\
&\overset{\eqref{pMT1}}{\leq} \sqrt n  K |\partial F| \int_t^{t_k} \|\bar H_s -H_s\|_{L^2(\partial E_s)} \ \d s \\
&\overset{\eqref{m1}}{\leq} \sqrt n K |\partial F| \|\psi_0\|_{H^3(\partial F)} \int_t^{t_k}  e^{-\sigma_1 s} \ \d s \\
&\leq  \frac{\sqrt n  K |\partial F|}{\sigma_1}\|\psi_0\|_{H^3(\partial F)} e^{-\sigma_1 t}.
\end{align*}
 Since \( D_{E_\infty} (E_{t_k}) \rightarrow 0\), then the previous estimate implies \eqref{pMT6} for \(t\). By doing a similar estimate
for \(D_F (E_{t_k})-D_F (E_0)\), using \eqref{pMT3} for \(F\) and Lemma \eqref{pMT7} and recalling that 
\(\|\psi_{t_k}\|_{H^3(\partial F)} \leq \epsilon_1\) we find a constant \(\tilde C_F\) not depending on the choice of \(\psi_0\) 
such that
\[
\|\psi_{t_k}\|_{L^\infty(\partial F)} \leq \tilde C_F \|\psi_0\|_{H^3(\partial F)}^{\frac{1}{2}(1-\frac{n-1}{6})}.
\]
Thus by passing to limit we see that \eqref{pMT7} holds for \(\psi_\infty\). Finally we check that the full \(C^{1,\frac{1}{4}}\)-convergence in time holds. To this end, it suffices to show that every sequence \( (\tilde t_k)_{k=1}^\infty \subset \mathbb R_+\) with \(\tilde t_k \rightarrow \infty\) has a subsequence converging to \(\psi_\infty\)
in \(C^{1,\frac{1}{4}}(\partial F)\). Indeed, by arguing as previously in a case of such sequence \( (\tilde t_k)_{k=1}^\infty \subset \mathbb R_+\), we find a subsequence converging to some limit \(\tilde \psi_\infty \in H^3(\partial F)\) in \(C^{1,\frac{1}{4}}(\partial F)\) and \(F_{\tilde \psi_\infty}\) is defined as a \(C^{1,\frac{1}{4}}\)-set.
We may again assume that the subsequence is the whole sequence and hence \( | E_{\tilde t_k} \Delta F_{\tilde \psi_\infty}| \rightarrow 0\), which implies together \eqref{pMT6} and the boundedness of \(\bar d_{E_\infty}\) that 
\[
D_{E_\infty}\left(F_{\tilde \psi_\infty}\right) = \lim_k D_{E_\infty}\left(E_{\tilde t_k}\right) = 0.
\]
This implies that \(E_\infty = F_{\psi_\infty}\) and further \(\psi_\infty = \tilde \psi_\infty\). Thus the first step has been concluded.
\newline
\newline
\emph{Proof of Step 2.} First we show that \(E_\infty\) is a smooth and critical set. Since \(E_\infty\) is a \(C^{1,\frac{1}{4}}\)-set, then thanks to Lemma \ref{Lwcrit} it suffices to show it to be stationary. We need to find  \(\lambda_\infty \in \mathbb R\) such that for every \(f \in C^\infty(\To^n;\mathbb R^n)\)  
\begin{equation}
\label{pMT8}
\int_{\partial E_\infty} \div_\tau f \ \d\H^{n-1} = \lambda_\infty \int_{\partial E_\infty}  \langle  f, \nu_\infty \rangle \ \d \H^{n-1},
\end{equation}
where \(\nu_\infty\) is the corresponding unit normal field of \(\partial E_\infty\) with inside-out orientation. 
Since \(\psi_t \rightarrow \psi_\infty\) in \(C^1(\partial F)\), then \( \Phi_{\psi_t} \rightarrow  \Phi_{\psi_\infty}\),  \(\nu_t \circ \Phi_{\psi_t} \rightarrow \nu_\infty \circ \Phi_{\psi_\infty}\) and \(J_\tau \Phi_{\psi_t} \rightarrow J_\tau \Phi_{\psi_\infty}\) uniformly on \(\partial F\).
Thus by using the change of variables formula we obtain for every \(f \in C^\infty(\To^n;\mathbb R^n)\)
\begin{align}
\label{pMT9a}
\int_{\partial E_t} \div_\tau f \ \d\H^{n-1} &\longrightarrow \int_{\partial E_\infty} \div_\tau f \ \d\H^{n-1} \ \ \ \text{and} \\
\label{pMT9b}
\int_{\partial E_t}  \langle  f, \nu_t \rangle \ \d\H^{n-1} &\longrightarrow \int_{\partial E_\infty}  \langle  f, \nu_\infty \rangle \ \d\H^{n-1}.
\end{align}
By using \eqref{pMT1}, \eqref{pMT2} and Hölder's inequality we see
%
\begin{comment}
\[
|\bar H_t| \leq \left(\H^{n-1}(\partial E_t)\right)^{-\frac{1}{2}}\|H_t\|_{L^2(\partial  E_t)} \leq \left(\frac{K}{\H^{n-1}(\partial F)}\right)^{\frac{1}{2}} K\left(\epsilon_1 + |H_{\partial F}|\right)
\]
\end{comment}
%
that \(\bar H_t\) is bounded in time and hence we find a sequence \((\bar H_{t_k})_{k}\), \(t_k \rightarrow \infty\), converging to some real number say \( \lambda_\infty\). 
By the divergence theorem
\begin{align*}
\int_{\partial E_{t_k}} \div_\tau f \ \d\H^{n-1} 
&= \int_{\partial E_{t_k}}  H_{t_k} \langle  f, \nu_t \rangle \ \d\H^{n-1} \\
&=   \bar H_{t_k} \int_{\partial E_{t_k}} \langle  f, \nu_{t_k} \rangle \ \d\H^{n-1} +  \int_{\partial E_{t_k}} (H_{t_k}- \bar H_{t_k}) \langle  f, \nu_t \rangle \ \d\H^{n-1} 
\end{align*}
and thus by letting \({t_k}\rightarrow \infty\) and recalling \eqref{pMT9a} and \eqref{pMT9b} we obtain \eqref{pMT8}, since 
\begin{align*}
\left|\int_{\partial E_{t_k}} (H_{t_k}- \bar H_{t_k}) \langle  f, \nu_{t_k} \rangle \ \d\H^{n-1}\right| 
&\leq \sup_{\To^n} |f| |\partial E_{t_k}|^\frac{1}{2} \|(H_{t_k}- \bar H_{t_k})\|_{L^2(\partial E_{t_k})}\\
&\overset{\eqref{pMT1}}{\leq} \sup_{\To^n} |f| \left(K |\partial F|\right)^\frac{1}{2} \|(H_{t_k}- \bar H_{t_k})\|_{L^2(\partial E_{t_k})}\\
&\overset{\eqref{m1}}{\leq} \sup_{\To^n} |f| \left(K |\partial F|\|\psi_0\|_{H^3(\partial F)} e^{-\sigma_1{t_k}}\right)^\frac{1}{2} .
\end{align*}
Thus \(E_\infty\) is a smooth and critical set and since \(\|\psi_\infty\|_{H^3(\partial F)} \leq \epsilon_1\) (recall the choice of \(\epsilon_1\)) and \(|E_\infty| = |F|\) (by \eqref{pMT4}), it follows from Lemma \ref{Ltransl0} that \(E_\infty = F + p\) for some \(p \in \mathbb R^n\).
Since now \(d_H(F,E_\infty) \leq \|\psi_\infty\|_{L^\infty(\partial F)}\), then it follows from \eqref{pMT7} that \(d_H(F,E_\infty) \rightarrow 0\) as \(\|\psi_0\|_{H^3(\partial F)}\) tends to zero.
This implies that we may choose \(p\) in such a way that simultaneously \(p \rightarrow 0\).
\newline
\newline
\emph{Proof of Step 3.}  Since now \(F+ p = E_\infty\) and  \(\|\psi_\infty\|_{C^1(\partial E)},\|\psi_t\|_{C^1(\partial E)} \leq c \), then by (iii) 
we may write \(\partial E_t\) as a smooth graph  in normal direction over \(\partial (F+p)\), i.e.,
for every \(t \in [0,\infty[\) there is a unique \(\varphi_t \in C^\infty(\partial E_\infty)\) for which \(E_t =(E_\infty)_{\varphi_t}\). 
Again, for every \(t \in [0,\infty[\)
\[
\|\varphi_t\|_{L^2(\partial E_\infty)}  \overset{\eqref{pMT3}}{\leq} K \sqrt{D_{E_\infty}(E_t)} \overset{\eqref{pMT6}}{\leq} K C_F^\frac{1}{2}  e^{-\frac{\sigma_1}{2} t}  
\]
and
\begin{align*}
\|\varphi_t\|_{H^3(\partial E_\infty)} 
&\overset{\eqref{pMT4}}{\leq} K\left(\|\nabla_\tau H_t\|_{L^2(\partial E_t)}+ \sqrt{D_{E_\infty}(E_t)}\right)\\
&\overset{\eqref{m2}}{\leq} K\left(\|\nabla_\tau H_0\|_{L^2(\partial E_0)} + C_0^\frac{1}{2}\|\bar H_0 - H_0\|_{L^2(\partial E_0)}+ \sqrt{D_{E_\infty}(E_t)}\right)\\
&\overset{\eqref{pMT4}}{\leq} K\left(K\|\psi_0\|_{H^3(\partial F)} + C_0^\frac{1}{2}\|\bar H_0 - H_0\|_{L^2(\partial E_0)}+ \sqrt{D_{E_\infty}(E_t)}\right)\\
&\overset{\eqref{m1}}{\leq} K\left( K\|\psi_0\|_{H^3(\partial F)}  +  C_0^\frac{1}{2} \|\psi_0\|_{H^3(\partial F)}^\frac{1}{2}+ \sqrt{D_{E_\infty}(E_t)}\right)\\
&\overset{\eqref{pMT6}}{\leq} K\left( K\|\psi_0\|_{H^3(\partial F)} +  C_0^\frac{1}{2} \|\psi_0\|_{H^3(\partial F)}^\frac{1}{2}+ C_F^\frac{1}{2}\|\psi_0\|_{H^3(\partial F)}^\frac{1}{2} \right)\\
& \leq K\left( K\epsilon_1^\frac{1}{2}+ C_0^\frac{1}{2}+C_F^\frac{1}{2} \right)\|\psi_0\|_{H^3(\partial F)}
\end{align*}
This means that there exists a positive constant \(C'_F\) independent of the choice of \(\psi_0\)
%satisfying \(\|\psi_0\|_{H^3(\partial F)} \leq \delta_0\) 
such that \(\|\varphi_t\|_{L^2(\partial E_\infty)} \leq C'_F \|\psi_0\|_{H^3(\partial F)}^\frac{1}{2}  e^{-\frac{\sigma_1}{2} t} \) and \(\|\varphi_t\|_{H^3(\partial E_\infty)}  \leq C'_F \|\psi_0\|_{H^3(\partial F)}^\frac{1}{2}\) for every \(t \in [0,\infty[\). Since 
\(\partial (F+p)\) shares same interpolation bounds than \(\partial F\), then by using the previous estimates and  Lemma \ref{Int1}
in the case
\begin{equation}
\label{pMT10}
\frac{1}{5} = \frac{j}{n-1} + \left(\frac{1}{2} - \frac{3}{n-1}\right)\alpha + \frac{1-\alpha}{2}
\end{equation}
with \(\alpha = \frac{j}{3} + \frac{n-1}{10}\), \(j=0,1,2\) and a corresponding interpolation constant (which we may assume to be the same \(C'_F\))
we obtain 
\begin{align*}
\|\nabla_\co^j\varphi_t\|_{L^5(\partial E_\infty)} 
&\leq C'_F \|\nabla_\co^3 \varphi_t\|_{L^2(\partial E_\infty)}^{\frac{j}{3} + \frac{n-1}{10}}  \|\varphi_t\|_{L^2(\partial E_\infty)}^{\frac{3-j}{3} - \frac{n-1}{10}}+ C_F' \|\varphi_t\|_{L^1(\partial E_\infty)} \\
&\leq C'_F\|\varphi_t\|_{H^3(\partial E_\infty)}^{\frac{j}{3} + \frac{n-1}{10}}  \|\varphi_t\|_{L^2(\partial E_\infty)}^{\frac{3-j}{3} - \frac{n-1}{10}} + C_F'|\partial F|^\frac12
\|\varphi_t\|_{L^2(\partial E_\infty)}\\
&\leq (C'_F)^2 \|\psi_0\|_{H^3(\partial F)}^\frac{1}{2} \left(e^{-\frac{\sigma_1}{2}\left(\frac{3-j}{3} - \frac{n-1}{10}\right)t} +  |\partial F|^\frac12 e^{-\frac{\sigma_1}{2} t}\right) \\
&\leq (1+|\partial F|^\frac12)(C'_F)^2 \|\psi_0\|_{H^3(\partial F)}^\frac{1}{2} e^{-\frac{\sigma_1}{2}\left(\frac{1}{3} - \frac{n-1}{10}\right)t} \\
&\overset{\eqref{pMT5}}{=} (1+|\partial F|^\frac12)(C'_F)^2 \|\psi_0\|_{H^3(\partial F)}^\frac{1}{2} e^{-\sigma_0t},
\end{align*}
which implies that there exists such \(C\) as claimed.
\end{proof}

Let us finally recall Remark \ref{MT2}. In the last step of the previous proof we may replace \(5\) in the left hand side of \eqref{pMT10} with any \(q\geq 1\) as long as the corresponding \(\alpha\) is strictly less than \(1\), because \(\sigma_0 = \frac{\eta_0}{2}(1- \alpha)\) would then be strictly positive.
In the case \(n=3\) by replacing \(5\) with any \(1 \leq q<\infty\) we obtain
\[
\alpha = 1 -\frac{2}{3q},
\]
so we see that any such \(q\) will do. Whereas in the case \(n=4\) doing so yields
\[
\alpha = \frac{7}{6} -\frac{1}{q}
\]
and hence \(q\) may take any values in the interval \([1,6[\).

%%%%%%%%%%%%%%%%%%%%%%%%%%%%%%%%%%%%%%%%%%%%%%%%%%%%%%%%%%%%%%%%%%%%%%%%%%%%%%%%%%%%%%%%%%%%%%%%%%%%%%%%%%%%%%%%%%%%%%

\section{Appendix}

\subsection*{\(C^1\)- and \(H^3\)-bounds}

In this subsection we prove the estimates \eqref{unifcontrol1a} and \eqref{unifcontrol1b}, Lemma \ref{LMC}, Lemma \ref{LC1unifcontrol}
and Lemma \ref{LH3unifcontrol}.
We will use the same notation as earlier without any further mention.
For sake of simplicity, we use the generic symbol \(C\) for a constant which may change from line to line
in the estimates but ultimately depends only on $E$.

Let us first fix a smooth set \(E\) and let \(U_E\) be a regular neighborhood of \(\partial E\) which we may assume to be of the form
\(U_E= \partial E + B(0,r)\) with some positive $r$.
 Recall that \(\bar d_E\) and \(\pi_{\partial E}\) 
are \(C^k\)-bounded in \(U_E\) for every \(k \in \mathbb N\). For every \(\psi \in C^\infty(\partial E)\) we set the smooth extension 
\(\psi_E = \psi \circ \pi_{\partial E}\). Then \(\nabla \psi_E = \nabla_\tau \psi\) on \(\partial E\) and
 moreover the following decomposition holds on \(\partial E\)
\begin{equation}
\label{decomp}
\D^2 \psi_E =-\nu_E \otimes B_E \nabla_\tau \psi - B_E \nabla_\tau \psi \otimes \nu_E + D^2_\tau \psi.
\end{equation}
For every \(\psi \in C^\infty(\partial E)\) we set \(\Phi_\psi: \partial E \rightarrow \To^n\), \(\Phi_\psi(x)=x+\psi\nu_E(x)\), as in Lemma \ref{Lc1equivalence}.
We extend \(\Phi_\psi\) to be the smooth map \(U_E \rightarrow \To^n\) given by \(\Phi_\psi (x) = x + \psi_E(x)\nabla \bar d_E(x)\). Then
\begin{align*}
\D\Phi_\psi = 
\begin{cases}
I + \psi_E \D^2\bar d_F + \nabla \bar d_F \otimes \nabla \psi_E  &\text{in} \ \ \ U_E \ \ \ \text{and} \\
 I + \psi B_E + \nu_E \otimes \nabla_\tau \psi  &\text{on} \ \ \ \partial E.
\end{cases}
\end{align*}
Now \(\Phi_\psi \rightarrow \id \) and \(\D\Phi_\psi \rightarrow I\) uniformly in \(U_E\) as \(\|\psi\|_{C^1(\partial E)} \rightarrow 0\). Thus, by possibly replacing \(U_E\) with \(\partial E + B(0,r/2)\), 
\(\Phi_\psi\) is an orientation preserving diffeomorphism from \(U_E\) to its image and the set \(E_\psi\) is well-defined, when 
\(\|\psi\|_{C^1(\partial E)}\) is small enough. \footnote{This implies the first part of Lemma \ref{Lc1equivalence} for smooth functions (the other cases are similar to check). The details of second part are 
left to the reader.} The inverse matrix on \(\partial E\) is then
\[
(\D\Phi_\psi)^{-1} = \left(I-\nu_E \otimes \nabla_\tau \psi\right) \left(I+\psi B_E\right)^{-1}.
\]
From now on, we assume \(\|\psi\|_{C^1(\partial E)} \leq \delta\) with \(\delta\) small enough so that the previous hold. 
Further, we use the shorthand notation \(A_\psi =  \left(I+\psi B_E\right)^{-1}\) on \(\partial E\). 
Set \(u_\psi : \Phi_\psi(U_E) \rightarrow \mathbb R\), \(u_\psi = \bar d_E \circ \Phi_\psi^{-1}\). Then \(\partial E_\psi = u_\psi^{-1}(0)\) and
\(\nu_{E_\psi} = \nicefrac{\nabla u_\psi}{|\nabla u_\psi|} \) on \(\partial E_\psi\). Again
\[
 \nabla u_\psi \circ \Phi_\psi =  (\D \Phi_\psi)^{-\T} \nu_E  =  \nu_E - A_\psi \nabla_\tau \psi \longrightarrow \nu_E 
\]
uniformly on \(\partial E\) as \(\|\psi\|_{C^1(\partial E)} \rightarrow 0\). Thus \(\nu_{E_\psi} \circ \Phi_\psi\) also converges uniformly to \(\nu_E\) as \(\psi\) goes to zero in \(C^1\)-sense. The second fundamental form on \(\partial E_\psi\) can be written, with help of \(u_\psi\), as
\begin{equation}
\label{graphB}
B_{E_\psi} = P_{\partial E_\psi} \D_\tau\left( \frac{\nabla u_\psi}{|\nabla u_\psi|}\right) 
= \frac{1}{|\nabla u_\psi|}(I-\nu_{E_\psi} \otimes \nu_{E_\psi})\D^2 u_\psi (I-\nu_{E_\psi} \otimes \nu_{E_\psi}).
\end{equation}
Omitting the details we may further compute that 
\begin{equation}
\label{clear}
\D^2 u_\psi \circ \Phi_\psi = A_\psi \Bigg[B_E -\psi  \left( \sum_{k=1}^n  \langle\nu_E - A_\psi \nabla_\tau \psi, v_k\rangle\partial_{v_k} \mathrm D^2 \bar d_E \right)  -\mathrm D^2 \psi_E \Bigg]A_\psi
\end{equation}
on \(\partial E\), where \(v_1,\ldots,v_n\) is any orthonormal basis of \(\mathbb R^n\). Hence  \eqref{decomp}, \eqref{graphB} and \eqref{clear} and the \(C^1\)-bound \(\delta\) imply that \(|\D^2 u_\psi \circ \Phi_\psi| \leq C(1+|\D_\tau^2 \psi|)\) on \(\partial E\) with some constant \(C\) and so \eqref{lMC2} holds. Again, by combining the expressions \eqref{graphB} and \eqref{clear} we may write on \(\partial E\)
\begin{equation}
\label{graphH}
H_{E_\psi} \circ \Phi_\psi = \tr \left(Q(\ \cdot \ , \psi,\nabla_\tau \psi )\Bigg[B_E -\psi  \left( \sum_{k=1}^n  \langle\nu_E - A_\psi \nabla_\tau \psi, v_k\rangle\partial_{v_k} \mathrm D^2 \bar d_E \right)  -\mathrm D^2 \psi_E \Bigg]\right),
\end{equation}
where \(Q: \partial E \times [-\delta,\delta] \times [-\delta,\delta]^n \rightarrow \mathcal L (\mathbb R^n;\mathbb R^n)\) is a smooth map with \(Q(\ \cdot \ , 0 ,0 ) = P_{\partial E}\). Thus
by using Taylor's expansion we may write on \(\partial E\)
\begin{equation}
\label{linexpansion}
Q(\ \cdot \ , \psi ,\nabla_\tau \psi ) = P_{\partial E} + \psi S(\ \cdot \ , \psi ,\nabla_\tau \psi ) + [\langle\nabla_\tau \psi, r_{ij}(\ \cdot \ , \psi ,\nabla_\tau \psi )\rangle]_{ij}
\end{equation}
with some smooth \(S\) and \(r_{ij}\). Thus by substituting  \eqref{linexpansion} and \eqref{decomp} to \eqref{graphH} we obtain the expression \eqref{lMC1} after regrouping the terms and Lemma \ref{LMC} is clear. 

Suppose that \(h \in L^p(\partial E_\psi)\) with \(1\leq p < \infty\) and \(\varphi \in C^\infty(\partial E_\psi)\). By using the change of variable formula we have
\begin{align*}
\int_{\partial E_\psi} |h|^p \ \d \H^{n-1} &= \int_E |h \circ \Phi_\psi|^p  J_\tau \Phi_\psi \ \d \H^{n-1}  \ \ \ \text{and}\\
\int_{\partial E_\psi} |\nabla_\tau \varphi|^p \ \d \H^{n-1} 
&= \int_{\partial E_\psi} |P_{\partial E_\psi} \nabla\left((\varphi \circ \Phi_\psi)_E \circ \Phi_\psi^{-1}\right) |^p \ \d \H^{n-1} \\
&= \int_{\partial E_\psi} |P_{\partial E_\psi} (\D\Phi_\psi^{-1})^{\T} \nabla (\varphi \circ \Phi_\psi)_E \circ \Phi_\psi^{-1} |^p \ \d \H^{n-1} \\
&= \int_{\partial E} |(I -  \nu_{E_\psi} \circ \Phi_\psi \otimes  \nu_{E_\psi} \circ \Phi_\psi ) (\D\Phi_\psi)^{-\T} \nabla_\tau (\varphi \circ \Phi_\psi) |^p  J_\tau \Phi_\psi \ \d \H^{n-1}. 
\end{align*}
Since now \((I -  \nu_{E_\psi} \circ \Phi_\psi \otimes  \nu_{E_\psi} \circ \Phi_\psi ) (\D\Phi_\psi)^{-\T} \rightarrow P_{\partial E}\) and \(J_\tau \Phi_\psi \rightarrow 1\) uniformly on \(\partial E\) as \(\|\psi\|_{C^1(\partial E)}\) tends to zero and \(P_{\partial E}\nabla_\tau (\varphi \circ \Phi_\psi) = \nabla_\tau (\varphi \circ \Phi_\psi)\), 
then by decreasing \(\delta\), if necessary, we find a uniform constant \(C\) such that
\begin{align*}
C^{-1} \|h \circ \Phi_\psi\|_{L^p(\partial E)}^p &\leq  \|h\|_{L^p(\partial E_\psi)}^p  \leq C \|h \circ \Phi_\psi\|_{L^p(\partial E)}^p \ \ \ \text{and} \\
C^{-1-p} \|\nabla_\tau (\varphi \circ \Phi_\psi)\|_{L^p(\partial E)}^p &\leq  \|\nabla_\tau \varphi\|_{L^p(\partial E_\psi)}^p  \leq C^{1+p} \|\nabla_\tau (\varphi \circ \Phi_\psi)\|_{L^p(\partial E)}^p,
\end{align*}
whenever \(\|\psi\|_{C^1(\partial E)} \leq \delta\). This establishes \eqref{unifcontrol1a} and \eqref{unifcontrol1b}. Again, Lemma \ref{LC1unifcontrol} is a direct consequence of 
\eqref{unifcontrol1a}, \eqref{unifcontrol1b} and Remark \ref{H3Remark}.

Since \(|\partial E_\psi|\) is now bounded, it is enough to prove Lemma \ref{LH3unifcontrol} for \(p=6\). To this end, choose an arbitrary 
\(\varphi \in C^\infty(\partial E_\psi)\) and define a smooth extension \(\varphi_{E_\psi} = \varphi \circ \pi_{\partial E_\psi}\) to some regular neighborhood of 
\(\partial E_\psi\) as before.
A straightforward but a rather long  computation 
yields
\begin{align*}
\D_\tau^2(\varphi \circ \Phi_\psi) 
&= P_{\partial E} (\D^2 \varphi_{E_\psi} \circ \Phi_\psi) \D \Phi_\psi P_{\partial E} + \langle \nu_E, \nabla_\tau \varphi \circ \Phi_\psi \rangle (\D_\tau^2 \psi- B_E) \\
&+\psi \sum_{i=1}^n P_{\partial E} e_i \otimes P_{\partial E} \left((\partial_i \D \bar d _E) (\nabla_\tau \varphi \circ \Phi_\psi) + (\D \Phi_\psi)^\T \D^2\varphi_{E_\psi}\circ \Phi_\psi \partial_i \nabla \bar d_E\right)\\
&+B_E (\nabla_\tau \varphi \circ \Phi_\psi) \otimes \nabla_\tau \psi + \nabla_\tau \psi \otimes B_E (\nabla_\tau \varphi \circ \Phi_\psi)+ \nabla_\tau \psi \otimes P_{\partial E} (\D\Phi_\psi)^\T(\D^2 \varphi_{E_\psi} \circ \Phi_\psi) \nu_E.
\end{align*}
 Hence with help of the previous identity
\begin{align}
\notag
|\D_\tau^2(\varphi \circ \Phi_\psi)| &\leq C|\D^2 \varphi_{E_\psi} \circ \Phi_\psi| + (C+ |\D_\tau^2 \psi|)|\nabla_\tau \varphi \circ \Phi_\psi| \\
\notag
&\overset{\eqref{decomp}}{\leq} C|\D^2_\tau \varphi \circ \Phi_\psi| + C\left(1+|B_{E_\psi} \circ \Phi_\psi|+ |\D_\tau^2 \psi|\right)|\nabla_\tau \varphi \circ \Phi_\psi| \\
\label{D2estimate}
&\overset{\eqref{lMC2}}{\leq} C|\D^2_\tau \varphi \circ \Phi_\psi| + C\left(1+ |\D_\tau^2 \psi|\right)|\nabla_\tau \varphi \circ \Phi_\psi|.
\end{align}

Again, Lemma \ref{Int1} is satisfied for \(\varphi\) with \(p=6\), \(r=q = 2\), \(j=1\), \(m=2\) and 
\(n=3,4\), where \(\alpha_6 = \frac{5}{6}\) for \(n=3\) and \(\alpha_6=1\) for \(n=4\). Then 
\begin{align*}
\| \nabla_\tau \varphi\|_{L^6(\partial E_\psi)}
&\overset{\eqref{unifcontrol1b}}{\leq} C \| \nabla_\tau (\varphi \circ \Phi_\psi)\|_{L^6(\partial E)}\\
&\overset{\text{Lemma \ }\ref{Int1}}{\leq}   C \| \D_\tau^2 (\varphi \circ \Phi_\psi)\|_{L^2(\partial E)}^{\alpha_6}\|\varphi \circ \Phi_\psi\|_{L^2(\partial E)}^{1-\alpha_6}\\
&\leq C \| \D_\tau^2 (\varphi \circ \Phi_\psi)\|_{L^2(\partial E)}+ C\| \varphi \circ \Phi_\psi\|_{L^2(\partial E)} \\
&\overset{\eqref{D2estimate}}{\leq}C \| \D_\tau^2 \varphi \circ \Phi_\psi\|_{L^2(\partial E)} +C\|\nabla_\tau \varphi \circ \Phi_\psi\|_{L^2(\partial E)}+ C\| \varphi \circ \Phi_\psi\|_{L^2(\partial E)}\\
&+C\left(\int_{\partial E}|D_\tau^2\psi|^2|\nabla_\tau \varphi \circ \Phi_\psi|^2 \ \d\H^{n-1}\right)^\frac{1}{2}\\
&\overset{\eqref{unifcontrol1a}+\eqref{unifcontrol1b}}{\leq}  C \| \D_\tau^2 \varphi \|_{L^2(\partial E_\psi)} + C\| \varphi \|_{H^1(\partial E_\psi)}\\
&+C\left(\int_{\partial E}|D_\tau^2\psi|^2|\nabla_\tau \varphi \circ \Phi_\psi|^2 \ \d\H^{n-1}\right)^\frac{1}{2} \\
&\leq C \|\D_\tau^2 \varphi \|_{L^2(\partial E_\psi)} + C\| \varphi \|_{H^1(\partial E_\psi)} +C \|D_\tau^2 \psi\|_{L^3(\partial E)} \|\nabla_\tau \varphi \circ \Phi_\psi \|_{L^6(\partial E)}\\
&\overset{\eqref{unifcontrol1a}}{\leq} C \|\D_\tau^2 \varphi \|_{L^2(\partial E_\psi)} + C\| \varphi \|_{H^1(\partial E_\psi)} +C \|D_\tau^2 \psi\|_{L^3(\partial E)} \|\nabla_\tau \varphi\|_{L^6(\partial E_\psi)}\\
&\overset{\eqref{llaplace1}}{\leq} C \|\Delta_\tau \varphi \|_{L^2(\partial E_\psi)} + C\| \varphi \|_{H^1(\partial E_\psi)} +C \|D_\tau^2 \psi\|_{L^3(\partial E)} \|\nabla_\tau \varphi\|_{L^6(\partial E_\psi)}  \\
&+C\left(\int_{\partial E_\psi}|B_{E_\psi}|^2|\nabla_\tau \varphi |^2 \ \d\H^{n-1}\right)^\frac{1}{2} \\
&\overset{\eqref{unifcontrol1a}}{\leq} C \|\Delta_\tau \varphi \|_{L^2(\partial E_\psi)} + C\| \varphi \|_{H^1(\partial E_\psi)} +C \|D_\tau^2 \psi\|_{L^3(\partial E)} \|\nabla_\tau \varphi\|_{L^6(\partial E_\psi)}  \\
&+C\left(\int_{\partial E}|B_{E_\psi} \circ \Phi_\psi|^2|\nabla_\tau \varphi \circ \Phi_\psi |^2 \ \d\H^{n-1}\right)^\frac{1}{2}\\
&\overset{\eqref{lMC2}}{\leq} C \|\Delta_\tau \varphi \|_{L^2(\partial E_\psi)} + C\| \varphi \|_{H^1(\partial E_\psi)} +C \|D_\tau^2 \psi\|_{L^3(\partial E)} \|\nabla_\tau \varphi\|_{L^6(\partial E_\psi)}  \\
&+C\left(\int_{\partial E}|\D_\tau^2 \psi|^2|\nabla_\tau \varphi \circ \Phi_\psi |^2 \ \d\H^{n-1}\right)^\frac{1}{2} +  C\|\nabla_\tau \varphi \circ \Phi_\psi \|_{L^2(\partial E)} \\
&\leq C \|\Delta_\tau \varphi \|_{L^2(\partial E_\psi)} + C\| \varphi \|_{H^1(\partial E_\psi)} +C \|D_\tau^2 \psi\|_{L^3(\partial E)} \|\nabla_\tau \varphi\|_{L^6(\partial E_\psi)}  \\
&+C \|D_\tau^2 \psi\|_{L^3(\partial E)} \|\nabla_\tau \varphi \circ \Phi_\psi \|_{L^6(\partial E)}+  C\|\nabla_\tau \varphi \circ \Phi_\psi \|_{L^2(\partial E)} \\
&\overset{\eqref{unifcontrol1b}}{\leq} C \|\Delta_\tau \varphi \|_{L^2(\partial E_\psi)} + C\| \varphi \|_{H^1(\partial E_\psi)} +C \|D_\tau^2 \psi\|_{L^3(\partial E)} \|\nabla_\tau \varphi\|_{L^6(\partial E_\psi)}.
\end{align*}
Again, by Remark \ref{H3Remark} \(\|\psi\|_{C^1(\partial E_\psi)},\|\D_\tau^2 \psi\|_{L^3(\partial E)} \leq C \|\psi\|_{H^3(\partial E_\psi)}\), when \(n\) is 3 or 4, so the previous estimate implies Lemma \ref{LH3unifcontrol} in the case \(p=6\).

\subsection*{Time derivatives}

In this subsection we derive the formulas \eqref{ltimederivatives1} and \eqref{ltimederivatives2}  of Lemma \ref{Ltimederivatives}. In particular, \eqref{ltimederivatives1} is probably well-known, but for sake of completeness we compute it too.

It follows from the \emph{semi-group property} that we need to check
\eqref{ltimederivatives1} and \eqref{ltimederivatives2} at the time \(t=0\).
 At first we list some facts. For that let \((E_t)_{t \in [0,T[}\) be any smooth flow in \(\To^n\) with a corresponding smoothly parametrized family of diffeomorphism 
\((\Phi_t)_{t \in [0,T[}\). Set the initial velocity vector field 
\[
X_0 = \partial_t \Phi_t \Big|_{t=0}.
\] 
Then \(V_0 = \langle X_0,\nu_0\rangle\) and the following hold on the initial boundary \(\partial E_0\).
\begin{align}
\label{dN1}
\frac{\partial}{\partial t}  \nu_t \circ \Phi_t \Big|_{t=0} &= -(\D_\tau X_0)^\T \nu_0,\\
\label{dJ1}
\frac{\partial}{\partial t} J_\tau\Phi_t\Big|_{t=0} &= \div_\tau X_0 \ \ \ \text{and}\\
\label{dH1}
\frac{\partial}{\partial t}  H_t \circ \Phi_t \Big|_{t=0} &= - \div_\tau \left( (\D_\tau X_0)^\T \nu_0 \right) - \tr \left( B_0 \D_\tau X_0\right).
\end{align}
For instance, \eqref{dN1} is directly computed in \cite{CMM}.
There are also an open neighborhood \(U\) of \(\bigcup_{t \in [0,T[} \partial E_t \times \{t\}\) and a smooth map \(H: U \rightarrow \To^n\) such that 
\(H(\ \cdot \ , t) = H_t\) on \(\partial E_t\) for every \(t \in [0,T[\).  Again, we recall that every smooth flow admits a \emph{normal parametrization}, which follows essentially from \cite[Theorem 8]{AD},
so we may assume \((\Phi_t)_{t \in [0,T[}\) to be such a parametrization. 
That is, 
\[
\partial_s \Phi_{t+s} \Big|_{s=0} = (V_t \circ \Phi_t)(\nu_t \circ \Phi_t)
\] 
on \(\partial E_0\) for 
every \(t \in [0,T[\), in particular \( X_0 = V_0\nu_0\) on \(\partial E_0\). Suppose from now on that \((E_t)_t\) is a volume preserving mean curvature flow, 
so we may write \(X_0 = (\bar H_0 - H_0)\nu_0\) and
\begin{equation}
\label{DtauX}
\D_\tau X_0 = (\bar H_0 - H_0) B_0 - \nu_0 \otimes \nabla_\tau H_0
\end{equation}
on \(\partial E_0\). Hence \eqref{dN1}, \eqref{dJ1} and \eqref{dH1} can be rewritten as
\begin{align}
\label{dN2}
\frac{\partial}{\partial t}  \nu_t \circ \Phi_t \Big|_{t=0} &= \nabla_\tau H_0,\\
\label{dJ2}
\frac{\partial}{\partial t} J_\tau\Phi_t\Big|_{t=0} &= (\bar H_0-H_0)H_0 \ \ \ \text{and}\\
\label{dH2}
\frac{\partial}{\partial t}  H_t \circ \Phi_t \Big|_{t=0} &= \Delta_\tau H_0 - (\bar H_0 - H_0)|B_0|^2.
\end{align}
The identity \eqref{dH2} can be also obtained in a more elegant way using the results from \cite{CMM}.
By using the change of variables formula and integration by parts we compute first
\begin{align*}
\frac{\d}{\d t} \|\bar H_t - H_t\|_{L^2(\partial E_t)}^2 \Big |_{t=0} 
= \frac{\d}{\d t} &\int_{\partial E_0} \left(\bar H_t \circ \Phi_t - H_t \circ \Phi_t \right)^2 J_\tau \Phi_t \ \d \H^{n-1} \Big |_{t=0} \\
= &\int_{\partial E_0} 2\left(\bar H_0 - H_0\right)\left(\frac{\partial}{\partial t}  \bar H_t \circ \Phi_t \Big |_{t=0}  - \frac{\partial}{\partial t}  H_t \circ \Phi_t \Big |_{t=0}  \right)  \ \d \H^{n-1}\\
+&\int_{\partial E_0} \left(\bar H_0 - H_0\right)^2  \frac{\partial}{\partial t}  J_\tau \Phi_t \Big |_{t=0} \ \d \H^{n-1} \\
= -2&\int_{\partial E_0}  \frac{\partial}{\partial t}  H_t \circ \Phi_t \Big |_{t=0}\left(\bar H_0 - H_0\right)  \ \d \H^{n-1}\\
+&\int_{\partial E_0} \left(\bar H_0 - H_0\right)^2  \frac{\partial}{\partial t}  J_\tau \Phi_t \Big |_{t=0} \ \d \H^{n-1} \\
\overset{\eqref{dJ2}+\eqref{dH2}}{=} -2&\int_{\partial E_0}  \left(\Delta_\tau H_0 - (\bar H_0 - H_0)|B_0|^2\right)\left(\bar H_0 - H_0\right)  \ \d \H^{n-1}\\
+&\int_{\partial E_0} H_0\left(\bar H_0 - H_0\right)^3 \ \d \H^{n-1} \\
=-2&\int_{\partial E_0} |\nabla_\tau H_0|^2 - |B_0|^2\left(\bar H_0 - H_0\right)^2 \ \d \H^{n-1}\\
+&\int_{\partial E_0} H_0\left(\bar H_0 - H_0\right)^3 \ \d \H^{n-1} \\
=-2&\partial^2P(\partial E_0)[\bar H_0 - H_0] +\int_{\partial E_0} H_0\left(\bar H_0 - H_0\right)^3 \ \d \H^{n-1}.
\end{align*}
To compute  \eqref{ltimederivatives2} at \(t=0\), we evaluate the term \(\frac{\partial}{\partial t}\left(\nabla_\tau H_t \circ \Phi_t \right)\Big|_{t=0}\) on \(\partial E_0\). We use the notation \(\nabla\) for the spatial gradient. Now
\begin{align*}
\frac{\partial}{\partial t}\left(\nabla_\tau H_t \circ \Phi_t \right)\Big|_{t=0} 
&= \frac{\partial}{\partial t} \left(I - \nu_t \circ \Phi_t \otimes\nu_t \circ \Phi_t \right) \nabla H (\ \cdot \ , t) \circ \Phi_t \Big |_{t=0} \\
&=-\left( \frac{\partial}{\partial t} \nu_t \circ \Phi_t \Big |_{t=0} \otimes\nu_0 + \nu_0 \otimes  \frac{\partial}{\partial t} \nu_t \circ \Phi_t \Big |_{t=0} \right) \nabla H (\ \cdot \ , 0) \\
&+\left(I - \nu_0 \otimes\nu_0 \right) \frac{\partial}{\partial t}  \nabla H (\ \cdot \ , t) \circ \Phi_t \Big |_{t=0} \\
&\overset{\eqref{dN2}}{=} - \langle \nu_0,\nabla H (\ \cdot \ , 0)\rangle \nabla_\tau H_0 - |\nabla_\tau H_0|^2 \nu_0 \\
&+\left(I - \nu_0 \otimes\nu_0 \right) \frac{\partial}{\partial t}  \nabla H (\ \cdot \ , t) \circ \Phi_t \Big |_{t=0} 
\end{align*}
and 
\begin{align*}
 \frac{\partial}{\partial t}  \nabla H (\ \cdot \ , t) \circ \Phi_t \Big |_{t=0} 
&= \frac{\partial}{\partial t} (\D \Phi_t)^{-\T} \nabla \left(H (\ \cdot \ , t) \circ \Phi_t\right) \Big |_{t=0}\\
&=\frac{\partial}{\partial t} (\D \Phi_t)^{-\T}\Big |_{t=0} \nabla H (\ \cdot \ , 0) +  \frac{\partial}{\partial t}\nabla \left(H (\ \cdot \ , t) \circ \Phi_t\right)\Big |_{t=0}\\
&=-\left(\D\frac{\partial}{\partial t} \Phi_t \Big |_{t=0} \right)^\T\nabla H (\ \cdot \ , 0) +   \nabla  \frac{\partial}{\partial t} \left(H (\ \cdot \ , t) \circ \Phi_t\right)\Big |_{t=0}\\
&=-\left(\D X_0 \right)^\T\nabla H (\ \cdot \ , 0) + \nabla  \frac{\partial}{\partial t} \left(H (\ \cdot \ , t) \circ \Phi_t\right)\Big |_{t=0}.
\end{align*}
By combining the two previous expressions
\begin{align}
\notag
\frac{\partial}{\partial t}\left(\nabla_\tau H_t \circ \Phi_t \right)\Big|_{t=0} 
&= - \langle \nu_0,\nabla H (\ \cdot \ , 0)\rangle \nabla_\tau H_0 - |\nabla_\tau H_0|^2 \nu_0 \\
\notag
&-\left(\D_\tau X_0 \right)^\T\nabla H (\ \cdot \ , 0) + \nabla_\tau  \frac{\partial}{\partial t} \left(H (\ \cdot \ , t) \circ \Phi_t\right)\Big |_{t=0} \\
\notag
&\overset{\eqref{DtauX}}{=} - \langle \nu_0,\nabla H (\ \cdot \ , 0)\rangle \nabla_\tau H_0 - |\nabla_\tau H_0|^2 \nu_0 \\
\notag
&-\left( (\bar H_0 - H_0) B_0 - \nu_0 \otimes \nabla_\tau H_0\right)^\T\nabla H (\ \cdot \ , 0) + \nabla_\tau  \frac{\partial}{\partial t} \left(H (\ \cdot \ , t) \circ \Phi_t\right)\Big |_{t=0}\\
\label{dgradH}
&= - |\nabla_\tau H_0|^2 \nu_0 -(\bar H_0 - H_0) B_0\nabla_\tau H_0 
+ \nabla_\tau  \frac{\partial}{\partial t} \left(H (\ \cdot \ , t) \circ \Phi_t\right)\Big |_{t=0}.
\end{align}

Thus by using the change of variables formula and integrating by parts we finally compute

\begin{align*}
\frac{\d}{\d t} \|\nabla_\tau H_t\|_{L^2(\partial E_t)}^2 \Big |_{t=0} 
= \frac{\d}{\d t} &\int_{\partial E_0} \langle \nabla_\tau H_t \circ \Phi_t,  \nabla_\tau H_t \circ \Phi_t  \rangle J_\tau \Phi_t \ \d \H^{n-1} \Big |_{t=0} \\
=2 &\int_{\partial E_0} \langle \frac{\partial}{\partial t} \nabla_\tau H_t \circ \Phi_t \Big |_{t=0},  \nabla_\tau H_0  \rangle \ \d \H^{n-1}  \\
+ &\int_{\partial E_0} |\nabla_\tau H_t|^2 \frac{\partial}{\partial t} J_\tau \Phi_t \Big |_{t=0}  \ \d \H^{n-1}\\
\overset{\eqref{dJ2}}{=}2 &\int_{\partial E_0} \langle \frac{\partial}{\partial t} \nabla_\tau H_t \circ \Phi_t \Big |_{t=0},  \nabla_\tau H_0  \rangle \ \d \H^{n-1}  \\
+ &\int_{\partial E_0} |\nabla_\tau H_t|^2 (\bar H_0-H_0)H_0  \ \d \H^{n-1}\\
\end{align*}
\begin{align*}
\overset{\eqref{dgradH}}{=}2 &\int_{\partial E_0} \langle  - (\bar H_0 - H_0)B_0 \nabla_\tau H_0 + \nabla_\tau  \frac{\partial}{\partial t}\left(H (\ \cdot \ , t) \circ \Phi_t\right)\Big |_{t=0},
\nabla_\tau H_0 \rangle \ \d \H^{n-1}  \\
+ &\int_{\partial E_0} |\nabla_\tau H_t|^2 (\bar H_0-H_0)H_0  \ \d \H^{n-1}\\
=-2 &\int_{\partial E_0} \Delta_\tau H_0 \frac{\partial}{\partial t}H_t \circ \Phi_t\Big |_{t=0} \ \d \H^{n-1}  \\
-2 &\int_{\partial E_0}(\bar H_0 - H_0) \langle \nabla_\tau H_0, B_0 \nabla_\tau H_0 \rangle \ \d \H^{n-1}  \\
+ &\int_{\partial E_0} |\nabla_\tau H_t|^2 (\bar H_0-H_0)H_0  \ \d \H^{n-1}\\
\overset{\eqref{dH2}}{=}-2 &\int_{\partial E_0} (\Delta_\tau H_0)^2 - (\bar H_0 - H_0)|B_0|^2\Delta_\tau H_0  \ \d \H^{n-1}  \\
-2 &\int_{\partial E_0}(\bar H_0 - H_0) \langle \nabla_\tau H_0, B_0 \nabla_\tau H_0 \rangle \ \d \H^{n-1}  \\
+ &\int_{\partial E_0} |\nabla_\tau H_t|^2 (\bar H_0-H_0)H_0  \ \d \H^{n-1}.
\end{align*}

%%%%%%%%%%%%%%%%%%%%%%%%%%%%%%%%%%%%%%%%%%%%%%%%%%%%%%%%%%%%%%%%%%%%%%%%%%%%%%%%%%%%%%%%%%%%%%%%%%%%%%%%%%%%%%%%%%%%%%

\section*{Acknowledgments}

The author is very thankful to Vesa Julin for many helpful discussions and advice. 
The work was supported by the Academy of Finland grant 314227.

%%%%%%%%%%%%%%%%%%%%%%%%%%%%%%%%%%%%%%%%%%%%%%%%%%%%%%%%%%%%%%%%%%%%%%%%%%%%%%%%%%%%%%%%%%%%%%%%%%%%%%%%%%%%%%%%%%%%%%

%Article:author,title,journal,volume,year,pages
%Book::author,title,series,publisher,year,address,edition

\end{document}